\documentclass [11pt]{amsart}
\usepackage {amsmath, amssymb, amscd, mathrsfs, graphicx, color, url,epsf}
\usepackage[all,knot,arc]{xy}
\usepackage[text={6.5in,9in},centering,letterpaper,dvips]{geometry}

\newtheorem {theorem}{Theorem}[section]
\newtheorem {lemma}[theorem]{Lemma}
\newtheorem {proposition}[theorem]{Proposition}

\newtheorem {conjecture}[theorem]{Conjecture}

\theoremstyle{remark}
\newtheorem {remark}[theorem]{Remark}

\newtheorem {assumption}[theorem]{Assumption}

\def\zz {{\mathbb{Z}}}
\def\rr {{\mathbb{R}}}

\def\qq {{\mathbb{Q}}}
\def\Q {\mathcal{Q}}
\def\ff {\mathbb{F}}
\def\X {{X}}
\def\Z {{Z}}
\def\L {\mathscr{L}}
\def\F {\mathcal{F}}
\def\R{\mathcal{R}}
\def\Sym{\mathrm{Sym}}

\newcommand\Ta{\mathbb{T}_\alpha}
\newcommand\Tb{\mathbb{T}_\beta}
\def\RR {\mathfrak{R}}
\def\KR {\operatorname{KR}}
\def\Bkr {\B_{\operatorname{KR}}}
\def\hBkr {\hB_{\operatorname{KR}}}
\def\Ckr {\C_{\operatorname{KR}}}

\def\Hkr {H_{\operatorname{KR}}} 
\def\bHkr {\overline{H}_{\operatorname{KR}}} 
\def\K {\mathcal{K}}
\def\Koszul {\mathscr{K} \! }

\def\B {\mathscr{B}}
\def\hB {\mathcal{B}}
\def\bH{\overline{H}}
\def\x {\mathbf{x}}
\def\y {\mathbf{y}}

\def\M {\mathcal{M}}
\def\del {\partial}
\def\XX {\mathcal{X}}
\def\YY {\mathcal{Y}}

\def\Malg {M_{\operatorname{alg}}}
\def\C {\mathcal{C}}
\def\U {\mathcal{U}}

\DeclareMathOperator{\rk}{rk}

\def\gr {\operatorname{gr}}
\def\tgr {\widetilde \gr}
\def\tg {{\Gamma}}
\def\Into {{\operatorname{in}}}
\def\Outof {{\operatorname{out}}}
\def\In {{\operatorname{In}}}
\def\Out {{\operatorname{Out}}}
\def\Tor {{\operatorname{Tor}}}
\def\tot { \operatorname{tot}}
\def\old { \operatorname{old}}
\def\new { \operatorname{new}}
\def\ein {{\In(W)}}
\def\eout {{\Out(W)}}

\def\ss {{\mathcal{S}}}
\def\hh {{\mathcal{H}}}

\begin{document}

\title{An untwisted cube of resolutions for knot Floer homology}

\author[Ciprian Manolescu]{Ciprian Manolescu}
\thanks {The author was partially supported by NSF grants DMS-0852439 and DMS-1104406.}
\address {Department of Mathematics, UCLA, 520 Portola Plaza\\ 
Los Angeles, CA 90095}
\email {cm@math.ucla.edu}

\begin{abstract}
Ozsv\'ath and Szab\'o gave a combinatorial description of knot Floer homology based on a cube of resolutions, which uses maps with twisted coefficients. We study the $t=1$ specialization of their construction. The associated spectral sequence converges to knot Floer homology, and we conjecture that its $E_1$ page is isomorphic to the HOMFLY-PT chain complex of Khovanov and Rozansky. At the level of each $E_1$ summand, this conjecture can be stated in terms of an isomorphism between certain Tor groups. As evidence for the conjecture, we prove that such an isomorphism exists in degree zero.
\end {abstract}

\maketitle

\section {Introduction}

Knot homology theories are among the most effective tools for studying knots in $S^3$. Roughly, knot homologies are of two types. The first have their origins in representation theory and quantum topology. Examples include Khovanov's categorification of the Jones polynomial, and 
Khovanov and Rozansky's categorification of the quantum $sl(n)$ polynomial and the HOMFLY-PT  polynomial; see  \cite{Khovanov}, \cite{KR1}, \cite{KR2}. The second type of knot homologies are those with origins in gauge theory and symplectic geometry. The most studied among these is the {\em knot Floer homology} of Ozsv\'ath-Szab\'o and Rasmussen (\cite{Knots}, \cite{RasmussenThesis}).

Knot Floer homology admits several purely combinatorial descriptions, some coming with appropriate combinatorial proofs of invariance; see \cite{MOS}, \cite{MOST}, \cite{SarkarWang}, \cite{CubeResolutions}, \cite{Gilmore}, \cite{BaldwinLevine}. The description given by Ozsv\'ath and Szab\'o in \cite{CubeResolutions} is the one closest in spirit to the usual definitions of the representation-theoretic knot homologies. It is based on establishing an exact triangle for knot Floer homology that involves singular links, and uses twisted coefficients. Given a braid diagram for a knot, an iteration of this triangle produces a spectral sequence, which is shown to collapse at the $E_2$ page. This page is then described combinatorially. There are in fact two variants of the spectral sequence, corresponding to the two variants of knot Floer homology denoted $\widehat{HFK}$ (whose graded Euler characteristic is the Alexander polynomial $\Delta_K(T)$) and $HFK^-$ (whose graded Euler characteristic is $\Delta_K(T)/(1-T)$) .

As mentioned in \cite{CubeResolutions}, if the maps in the spectral sequences were untwisted, the results would look very similar to the HOMFLY-PT homology of Khovanov and Rozansky from \cite{KR2}. The purpose of this paper is to give evidence for a precise conjecture connecting the untwisted spectral sequences and HOMFLY-PT homology.

To fix notation, let $K$ be an oriented knot in $S^3$ with a decorated braid projection $\K$, as in \cite{CubeResolutions}. Specifically, $\K$ consists of a braid diagram drawn vertically, with the strands oriented upwards, and closed up by taking the top strands around to the right of the braid, so that the resulting planar diagram represents the knot $K$. Further,  in \cite{CubeResolutions} one of the leftmost edges in the braid is distinguished; for convenience, we will always take this to be the top leftmost one, which is one of the strands closed up when taking the braid closure (and thus is also the rightmost edge in the planar diagram). See Figure~\ref{fig:fig8} for an example. Let $c(\K)$ be the set of crossings in $\K$, and let $n$ be the number of crossings. We denote by $E = \{e_0, \dots, e_{2n} \}$ the set of edges in the diagram, where the distinguished edge is viewed as subdivided in two. We choose the ordering of the edges so that $e_0$ is the second segment on the distinguished edge, according to the orientation of $\K$. 

\begin {figure}
\begin {center}
\input {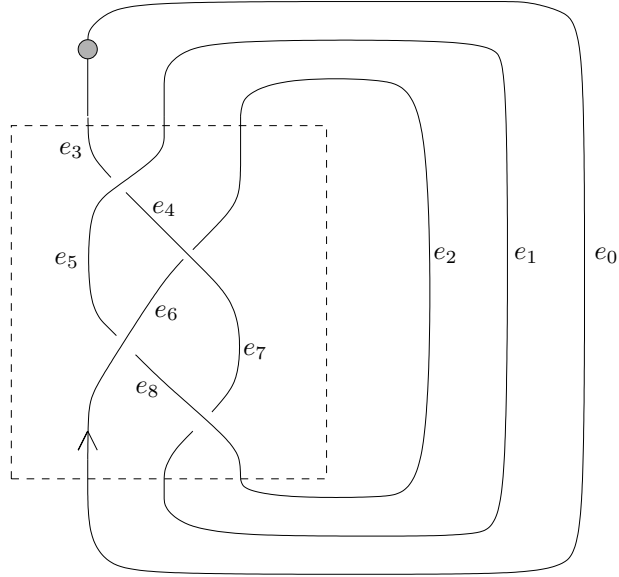}
\caption {A decorated braid projection for the figure-eight knot. The gray dot marks the subdivision of the distinguished edge.}
\label {fig:fig8}
\end {center}
\end {figure}

If $p$ is a crossing in $\K$, we define the {\em smoothing} of $\K$ at $p$ to be its oriented resolution  at $p$ (which is a link diagram with one fewer crossing), with two valence two vertices added, one on each side of where the crossing was. We also define the {\em singularization} of $\K$ at $p$ to be the diagram obtained from $\K$ by replacing the crossing at $p$ with a double point (resulting in a diagram for a singular link). A {\em complete resolution} $S$ of $\K$ is a diagram obtained from $\K$ by assigning smoothings or singularizations to all crossings. Thus, there are $2^n$ possible complete resolutions; each of them is a planar graph with vertices of valence either two or four. The point where the distinguished edge is subdivided (the gray dot in Figure~\ref{fig:fig8}) is not a vertex, but rather a place where the edge is cut open into two segments.

Let $\R$ be the polynomial algebra $\zz[U_0, \dots, U_{2n}]$. Each variable $U_i$ corresponds to an  edge $e_i \in E$. Given a complete resolution $S$ of $\K$, we will define an $\R$-module $\hB(S)$ as follows. 

First, let $c(S) \subseteq c(\K)$ be the subset of crossings of $\K$ that were singularized in $S$ (that is, the set of four-valent vertices in the graph of associated to $S$). At any $p \in c(S)$, if we denote by $a$ and $b$ the two outgoing edges, and by $c$ and $d$ the two incoming edges (as in Figure~\ref{fig:res}), we define the element
\begin {equation}
\label {eq:Lofp}
 L(p) = U_a + U_b - U_c - U_d \in \R.
 \end {equation}
We denote by $L_S \subset \R$ the ideal generated by all the elements $L(p)$ for $p \in c(S)$.

\begin {figure}
\begin {center}
\input {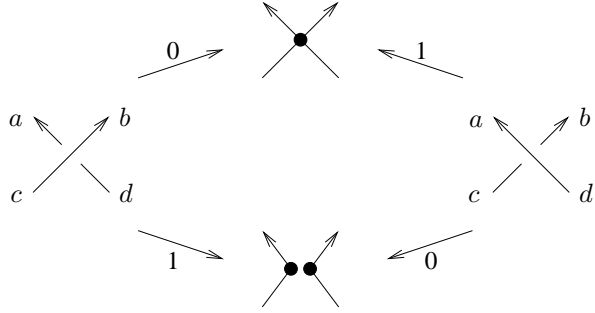}
\caption {The top picture represents the singularization of a crossing, and the bottom picture the smoothing. The smoothing and the singularization are also called the 0- and 1-resolutions of the crossing; which is which depends on whether the original crossing is positive or negative, as shown in the figure.}
\label {fig:res}
\end {center}
\end {figure}
 
Next, let $W$ be a collection of  vertices in the graph of $S$. (Here, the two loose ends that result from cutting the distinguished edge are not considered vertices.) We denote by $\Into (W)$ and $\Outof(W)$ be the sets of incoming and outgoing edges of $W$. We set
$$ \In(W) = \Into(W) \setminus \Outof(W), \ \Out(W) = \Outof(W) \setminus \Into (W),$$
and define $N_S \subset \R$ to be the ideal generated by the elements
$$ N(W) = \prod_{e \in \Out(W)} U_e  \ -  \prod_{e \in \In(W)} U_e,$$
over all possible collections $W$. 
 
Let also $V_S$ be the free $\R$-module spanned by the connected components of $S$ that do not contain the edge $e_0$. Let $\Lambda^*V_S$ be the exterior algebra of $V_S$. We define
\begin {equation}
\label {eq:hb}
 \hB(S) := \Tor_*(\R/L_S, \R/N_S) \otimes_{\R} \Lambda^* V_S,
 \end {equation}
where the $\Tor$ groups are taken over $\R$. 

Next, we organize all the complete resolutions of $\K$ into a hypercube, as in \cite{KR1}, \cite{KR2}, \cite{CubeResolutions}. If $p \in c(\K)$ is a positive crossing, we define the 0-resolution of $K$ at $p$ to be its singularization at $p$, and its 1-resolution to be the smoothing at $p$. If $p$ is a negative crossing, we let the $0$-resolution be the smoothing and its 1-resolution the singularization. (See Figure~\ref{fig:res}.) With these conventions, for any assignment $I: c(\K) \to \{0,1\}$, we obtain a complete resolution $S_I(\K)$.

Consider the direct sum
$$ \C(\K) = \bigoplus_{I: c(\K) \to \{0,1\}}  \hB(S_I(\K)).$$

Let $\ff = \zz/2\zz$. In \cite{CubeResolutions}, Ozsv\'ath and Szab\'o built a spectral sequence using modules over the base ring $\R \otimes_\zz \ff[t^{-1}, t]]$, where $\ff[t^{-1}, t]]$ is the field of half-infinite Laurent power series. Specializing to $t=1$, their result reads as follows:

\begin {theorem}[Ozsv\'ath-Szab\'o \cite{CubeResolutions}]
\label {thm:os}
Let $K \subset S^3$ be an oriented knot, and $\K$ a decorated braid projection of $K$, as above. 

$(a)$ There is a spectral sequence whose $E_1$ page is isomorphic to $\C(\K) \otimes_\zz \ff$, and which converges to the knot Floer homology $HFK^-(K)$  with coefficients in $\ff$.

$(b)$ There is a spectral sequence whose $E_1$ page is isomorphic to ${\C}(\K)/(U_0 = 0) \otimes_\zz \ff$, and which converges to the knot Floer homology $\widehat{HFK}(K)$ with coefficients in $\ff$.
\end {theorem}

Theorem~\ref{thm:os} is expected to hold also with coefficients in $\zz$ rather than $\ff$, but at the moment the orientations for link Floer complexes are not fully worked out in the literature.

An important difference between the spectral sequences in Theorem~\ref{thm:os} (with untwisted coefficients) and the original ones in \cite{CubeResolutions} (with twisted coefficients) is that the latter collapse at the $E_2$ stage for grading reasons. Because of this property, Ozsv\'ath and Szab\'o were able to use their spectral sequences to give their combinatorial descriptions of the knot Floer homology groups $HFK^-$ and $\widehat{HFK}$. In the untwisted setting, the $E_2$ terms do not typically live in a single grading, so we do not expect the sequences to collapse. On the other hand, we do expect an interesting relationship with the HOMFLY-PT homology of Khovanov and Rozansky, as follows.

Let us discuss some aspects of the construction of the HOMFLY-PT homology from \cite{KR2}. In the original reference Khovanov and Rozansky worked with coefficients in $\qq$, but the HOMFLY-PT homology can also be constructed with $\zz$ coefficients, as shown by Krasner \cite{Krasner}. We choose a decorated braid projection $\K$ for a knot $K$, as before. Given a complete resolution $S$ of $\K$, one associates to 
$S$ a Koszul complex $\Bkr(S)$. The HOMFLY-PT chain complex is then defined as
\begin {equation}
\label {eq:krypton}
 \Ckr (\K) =  \bigoplus_{I: c(\K) \to \{0,1\}}  H_*(\Bkr(S_I(\K))),
 \end {equation}
with a differential given by summing up certain zip and unzip maps. We denote its homology by $\Hkr(K)$. This is the {\em middle HOMFLY-PT homology} of the knot $K$. If we take the homology of $\Ckr(\K)/(U_0 = 0)$ instead, we obtain another variant of HOMFLY-PT homology, called {\em reduced}, which we denote by $\bHkr(K)$. (The terminology {\em middle} and {\em reduced} was introduced by Rasmussen \cite{RasmussenD}.)

We can alternately describe the summands in \eqref{eq:krypton} as follows. For a connected complete resolution $S$, let $Q_S \subset \R$ be the ideal generated by the quadratic elements
$$ Q(p) = U_a U_b - U_c U_d, $$
for all four-valent vertices $p \in c(S)$, together with the linear elements 
 $$ Q(p) = U_e - U_f,$$ 
 for all two-valent vertices $p$ of $S$, where $e$ and $f$ denote the edges meeting at $p$. For a disconnected complete resolution $S$, on each connected component that does not contain the distinguished edge in $\K$ we pick a two-valent vertex, coming from the right hand side of a resolved crossing. We call these two-valent vertices {\em special}, and define an ideal $Q_S \subset \R$ the same way as in the connected case, except that when $p$ is special we do not include the linear element $Q(p)$ in the generator set. (This is equivalent to cutting edges open at the special points, just as we did at the gray dot in Figure~\ref{fig:fig8}.)
 
This way, we have an ideal $Q_S$ for any complete resolution $S$. Observe that $Q_S$ is contained in the ideal $N_S$ defined previously. We will prove:
 
\begin {theorem}
\label {prop:kr}
For any complete resolution $S$ of a decorated braid diagram $\K$, the homology $H_*(\Bkr(S))$ is isomorphic to 
 \begin {equation}
\label {eq:hbkr}
 \hBkr(S) := \Tor_*(\R/L_S, \R/Q_S) \otimes_{\R} \Lambda^* V_S.
 \end {equation}
\end {theorem}

The expression \eqref{eq:hbkr} is similar to that for $\hB(S)$ from \eqref{eq:hb}. In fact, we propose the following:

\begin {conjecture}
\label {conj1}
Let $K \subset S^3$ be an oriented knot, with a decorated braid projection $\K$.

$(a)$ For every complete resolution $S$ of $\K$, the $\R$-modules $\hB(S)$ and $\hBkr(S)$ are isomorphic. 

$(b)$ Further, after tensoring with $\ff$, the isomorphisms in $(a)$ commute with the differentials on the complexes $\C(\K) \otimes_{\zz} \ff$ and $\Ckr(\K) \otimes_{\zz} \ff$, where on $\C(K) \otimes_{\zz} \ff$ we use the $d_1$ differentials from the spectral sequences in Theorem~\ref{thm:os}. As a consequence, the $E_2$ page of the spectral sequence from Theorem~\ref{thm:os} (a) is isomorphic to the middle HOMFLY-PT homology $\Hkr(K) \otimes_\zz \ff$, and the $E_2$ page of the spectral sequence from Theorem~\ref{thm:os} (b) is isomorphic to the reduced HOMFLY-PT homology $\bHkr(K) \otimes_{\zz} \ff$.
\end {conjecture}

Let us put this conjecture into context. A relationship between the HOMFLY-PT and knot Floer homology was first proposed by Dunfield, Gukov, and Rasmussen in \cite{DGR}, where they suggested the existence of a differential $d_0$ on $\bHkr(K)$, such that the homology with respect to $d_0$ gives $\widehat{HFK}(K)$. In light of Rasmussen's work in \cite{RasmussenD}, it became more natural to expect a spectral sequence from $\bHkr(K)$ to $\widehat{HFK}(K)$. Conjecture~\ref{conj2} (b), together with Theorem~\ref{thm:os}, would provide such a spectral sequence, at least with $\ff$ coefficients. Its existence would show that the total rank of the (reduced) HOMFLY-PT homology is at least as big as that of knot Floer homology. In turn, this would give a new proof of the fact that HOMFLY-PT homology detects the unknot. Currently, this last fact is known due to the work of Kronheimer and Mrowka \cite{KMsutured, KMunknot}, combined with that of Rasmussen \cite[Theorem 2]{RasmussenD}. Moreover, the existence of the spectral sequence (with the expected behavior with respect to gradings) would go beyond unknot detection: for example, it would show that HOMFLY-PT homology detects the two trefoils and the figure-eight knot, by using the corresponding result in knot Floer homology \cite{Ghiggini}.

The current paper reduces the Dunfield-Gukov-Rasmussen conjecture to a statement in terms of Tor groups, Conjecture~\ref{conj1}, of which part (a) has a purely algebraic flavor involving only ideals associated to graphs in the plane. This makes part (a) amenable to techniques from commutative algebra. Further, it is natural to expect that any solution to part (a) would produce isomorphisms that behave well with respect to the differentials, hence proving part (b).  

Thus, let us focus on part (a) of Conjecture~\ref{conj1}. Given how $\hB(S)$ and $\hBkr(S)$ are described in \eqref{eq:hb} and \eqref{eq:hbkr}, this part boils down to an isomorphism between $\Tor$ groups. A natural strategy of attacking Conjecture~\ref{conj1} (a) would be to cut the braid into simpler pieces and use an inductive argument. Although we have not succeeded in implementing this strategy, it is hopeful that the following extension of Conjecture~\ref{conj1} (a) seems to hold. 

Define a {\em partial braid graph} $S$ to be a part of a complete resolution $S'$ of a decorated braid projection. Precisely, let $W'$ be the set of crossings of $S'$, and view $S'$ as a union of neighborhoods $U_p$ of each $p\in W'$, such that $U_p$ consists of two segments intersecting at $p$. Then, at each $p$, do one of the following:

- keep $U_p$ as it is;

- delete one of the two segments in $U_p$, and either keep $p$ as a vertex, or erase it;

- delete both of the segments in $U_p$, together with $p$.

The result, $S$, is what we call a partial braid graph. It consists of a set of vertices $W \subseteq W'$, together with a set of edges. In $S$, an edge $e_i$ may have only one endpoint at a vertex in $W$; if so, we say that $e_i$ is an {\em exterior edge}, and do not consider its other endpoint to be a true  vertex of $S$. In particular, the original distinguished edge is  split into two exterior edges. With these conventions, we can define ideals $L_S, N_S,$ and $Q_S$ just as before. For simplicity in defining $Q_S$, let us assume that $S$ is connected. We also assume that $S$ contains at least one (hence at least two) exterior edges. (Note that Conjecture~\ref{conj1} (a) for general complete resolutions $S$ would follow from the case of connected $S$ with the distinguished edge cut open. Thus, it is natural to make a similar assumption on partial braid graphs.)

\begin {conjecture}
\label {conj2}
If $S$ is a connected partial braid graph with at least one exterior edge, then for any $i \geq 0$ we have an isomorphism of $\R$-modules
$$ \Tor_i(\R/L_S, \R/N_S) \cong \Tor_i(\R/L_S, \R/Q_S).$$
\end {conjecture}

An example of a partial braid graph is shown in Figure~\ref{fig:partial}, where
\begin {eqnarray*}
 L_S &=& (U_a + U_b - U_c - U_d, \ U_e + U_d - U_f - U_g), \\
 Q_S &=& ( U_a U_b - U_c  U_d, \ U_e U_d - U_f  U_g, \ U_g - U_b), \\
N_S &=& Q_S + (U_a U_e - U_c U_f  ). 
\end {eqnarray*}

\begin {figure}
\begin {center}
\input {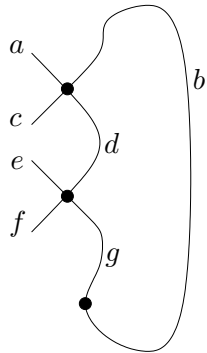}
\caption {A partial braid graph with four exterior edges and three vertices.}
\label {fig:partial}
\end {center}
\end {figure}

Evidence for Conjecture~\ref{conj2} comes from computer experimentation, and from proofs in some particular cases. For example, we have:

\begin {theorem}
\label {thm:2}
Conjecture~\ref{conj2} is true for $i=0$. In fact, for any connected partial braid graph $S$, we have $L_S + N_S = L_S + Q_S$ as ideals in $\R$. 
\end {theorem}

Interestingly, the isomorphism appearing in Conjecture~\ref{conj2} cannot be simply induced by the natural quotient map $\R/Q_S \to \R/N_S$ (although this is the case for $i=0$). See Section~\ref{sec:fail} for a discussion.

This paper is organized as follows. In Section~\ref{sec:sps} we review the proof of Theorem~\ref{thm:os}, focusing on the few aspects that are different in the untwisted setting. In Section~\ref{sec:kr} we prove Theorem~\ref{prop:kr}, about the HOMFLY-PT complex. In Section~\ref{sec:gr} we present and compare three gradings on the complexes $\C(\K)$ and $\Ckr(K)$. Finally, in Section~\ref{sec:partial} we discuss Conjecture~\ref{conj2} and prove Theorem~\ref{thm:2}.  

\medskip
\noindent {\bf Acknowledgements.} The author wishes to thank Brian Conrad, Mark Green, Tye Lidman, Peter Ozsv\'ath, Jacob Rasumssen, and Zolt\'an Szab\'o for several helpful conversations during the course of this work. Clearly, this paper is very much influenced by the work of Ozsv\'ath and Szab\'o \cite{CubeResolutions}, where the original twisted cube of resolutions is constructed, and its specialization to $t=1$ is suggested.

\section {The untwisted spectral sequence}
\label {sec:sps}
For completeness, in this section we sketch the construction of the spectral sequences in Theorem~\ref{thm:os}, following \cite{CubeResolutions}. The original reference used a coefficient ring of Laurent power series in a variable $t$. Here we specialize to $t=1$, and this requires us to address a few (minor) additional points. Precisely, some care needs to be taken to make sure that the sums involved in the construction remain finite when setting $t=1$; this is an admissibility issue, and is settled in Lemma~\ref{lem:adm} below. Another small discussion is  needed for disconnected resolutions---see Lemma~\ref{lem:disjoint} below. These lemmas are the new content in this section. Apart from that, the constructions are due to Ozsv\'ath and Szab\'o, and our exposition follows \cite{CubeResolutions} closely (except for a few differences in notation and terminology).

Let $\ff = \zz/2\zz$. Let $\K$ be a decorated braid projection with $n$ crossings. We denote our base ring by 
$$\RR = \R \otimes_{\zz} \ff \cong \ff[U_0, \dots, U_{2n}].$$ By a slight abuse of notation, the ideals 
$L_S$ and $N_S$ from the Introduction are denoted the same way here, even though they are implicitly tensored with $\ff$. This is the only section of the paper where we have to work with $\ff$ coefficients; we will return to $\zz$ coefficients starting in Section~\ref{sec:kr}. 

Unless otherwise noted, all the tensor products in this section are taken over $\RR$.

\subsection {Floer complexes from planar diagrams}
\label {sec:planar}
Given an assignment $I: c(\K) \to \{0,1,\infty\}$, we can define a {\em partial resolution} $S = S_I(\K)$ of the decorated  braid projection $\K$ as follows. At each $p \in c(\K)$, we take the 0-resolution if $I(p)=0,$ the 1-resolution if $I(p) = 1$, and we leave the crossing as it is if $I(p)=\infty$. We denote by $\sigma$  the number of crossings $p$ such that $S$ has a singularization (a four-valent vertex) at $p$.

There is a unified way of constructing Heegaard diagrams (and Floer chain complexes) for all the partial resolutions $S.$ Following \cite[Section 4]{CubeResolutions}, near each crossing $p \in c(\K)$ we draw a local picture as in Figure~\ref{fig:exact}. If $I(p) = \infty$ and $p$ is a positive crossing of $\K$, we place two $X$ markings at the spots $A^0$ and $A^+$ in the figure.  If $I(p) = \infty$ and $p$ is a negative crossing, we place two $X$ markings at $A^0$ and $A^-$. If $I(p) \in \{0,1\}$ and $S$ is smoothed at $p$, we place $X$ markings at the two spots indicated by $B$. If $I(p) \in \{0,1\}$ and $S$ is singularized at $p$, we delete $\alpha_1$ and $\beta_1$ from the diagram (that is, we only draw the circles $\alpha_2$ and $\beta_2$), and place two $X$ markings in the middle bigon that is the intersection of the two disks with boundaries $\alpha_2$ and $\beta_2$. In all cases, we also place an $O$ marking on each edge of the diagram; in the figure, the circles marked $a, b, c, d$ correspond to the $O$ markings on the four edges meeting at $p$. Finally, we add a point at infinity to the plane to obtain $S^2 = \rr^2 \cup \{\infty\}$, and we place an additional $X$ marking at infinity. We draw the curves in the diagram so that the point where we cut the distinguished edge (the gray dot in Figure~\ref{fig:fig8}) can be joined to infinity by a path that does not intersect any of the alpha or beta curves.

\begin {figure}
\begin {center}
\input {exact.pstex_t}
\caption {The planar diagrams at a crossing.}
\label {fig:exact}
\end {center}
\end {figure}

In the end, for each $I$ we obtain a collection of $2n-\sigma$ alpha curves and $2n-\sigma$ beta curves on the sphere, together with $2n+1$ $X$-markings and $2n+1$ $O$-markings. This is a {\em balanced Heegaard diagram} for the singular link $S \subset S^3$, in the sense of \cite{OSSsingular}. 

Let $\Ta$ (resp. $\Tb$) be the tori in the symmetric product $\Sym^{2n-\sigma}(S^2)$ gotten by taking the product of all alpha (resp. beta) curves. For $\x,\y \in \Ta \cap \Tb$, we denote by $\pi_2(\x, \y)$ the space of relative homology classes of Whitney disks from $\x$ to $\y$ with boundaries on $\Ta, \Tb.$ For $\phi \in \pi_2(\x, \y)$, we let $\mu(\phi) \in \zz$ be its Maslov index, and we let $\widehat{M}(\phi)$ be the moduli space of pseudo-holomorphic disks (flow lines) in the class $\phi$, modulo reparametrization by $\rr$. We choose orderings of the markings as $X_0, \dots,  X_{2n}$ and $O_0, \dots, O_{2n}$. We let $X_i(\phi)$ resp. $O_i(\phi)$ be the local multiplicity of (the domain of) $\phi$ at $X_i$ resp. $O_i$. Further, at each singular point (four-valent vertex) on $S$ we have two $X$-markings in the same bigon; we denote the local multiplicity of $\phi$ in that bigon by $XX_j(\phi)$, for some $j=1, \dots, \sigma$. 

One can assign to any intersection point $\x \in \Ta \cap \Tb$ an {\em Alexander grading} $A(\x) \in \zz$ and a {\em Maslov grading} $M(\x) \in \zz$. (In fact, when $S$ has multiple components, there are several Alexander gradings, but here we just consider their sum.) We refer to \cite[Section 2.3]{CubeResolutions} for the exact definitions of $A$ and $M$, but let us mention that up to a shift, the gradings are determined by the following properties: for any $\x, \y$, and $\phi \in \pi_2(\x, \y)$, 
$$ A(\x) - A(\y) = \sum_{i=0}^{2n} \bigl( X_i(\phi) - O_i (\phi)), $$
and
$$ M(\x) - M(\y) = \mu(\phi) - 2 \sum_{i=0}^{2n} O_i(\phi) + 2 \sum_{j=1}^{\sigma} XX_j(\phi).$$

The {\em Floer chain complex} $CFL^-(S) = CF^-(\Ta, \Tb)$ is defined as follows. As a module over $\RR =\ff[U_0, \dots, U_{2n}]$, it is freely generated by the intersection points $\x \in \Ta \cap \Tb.$ As such, it comes with a bigrading $(A, M)$ induced by the one on generators, where a variable $U_i$ is set to be in bigrading $(-1, -2)$. The differential on $CFL^-(S)$ is given by
\begin {equation}
\label {eq:delhfk}
\del \x = \sum_{\y \in \Ta \cap \Tb} \sum_{\{\phi\in \pi_2(\x, \y) \mid \mu(\phi) = 1; X_i(\phi) = 0, \forall i\} } 
\# \widehat{\M}(\phi) \cdot \prod_{i=0}^{2n} U_i^{O_i(\phi)} \cdot \y.
\end {equation}

In order for the differential to be well-defined, we need to make sure that the sum in \eqref{eq:delhfk} is finite. This is guaranteed to be the case if the Heegaard diagram we use is {\em admissible} in the following sense. (Compare \cite{Links}, \cite{OSSsingular}, \cite{CubeResolutions}.) A {\em periodic domain} is a two-chain on the Heegaard surface whose boundary is a $\zz$-linear combination of alpha and beta curves, and whose multiplicity at each marking $X_i$ or $O_i$ is zero. The diagram is said to be admissible if every non-trivial periodic domain has both positive and negative multiplicities somewhere on the diagram.

\begin {lemma}
\label {lem:adm}
The planar Heegaard diagram for a partial resolution $S$, as constructed above (using $n$ copies of Figure~\ref{fig:exact}) is admissible. 
\end {lemma}

\begin {proof}
The argument is different from the one in \cite[Lemma 3.3]{CubeResolutions}, where the use of extra markings (to define twisted coefficients) made admissibility more transparent.

Let $p \in c(\K)$ be a crossing. We denote by $\pi_+(b)$ and $\pi_-(c)$ the disks (ovals) with boundaries $\alpha_1$ resp. $\beta_1$ in Figure~\ref{fig:exact}, containing the small circles marked $b$ resp. $c$. We also denote by $\pi_+(a)$  the annulus bounded by $\alpha_1$ and $\alpha_2$ in Figure~\ref{fig:exact}, containing $a$. Similarly $\pi_-(d)$ is the annulus bounded by $\beta_1$ and $\beta_2$ and containing $d$.

Let $S_0, S_1, \dots, S_\ell$ be the connected components of $S$, where $S_0$ contains the distinguished edge $e_0$. For each $i=1, \dots, \ell$, we have a periodic domain 
$$ \pi_i = \sum_{e \subseteq S_i} \bigl( \pi_+(e) - \pi_-(e) \bigr).$$
The set $\{\pi_i | \ i=1, \dots, \ell \}$ forms a basis for the space of periodic domains. On each edge $e \in S_i$, the multiplicity of $\pi_i$ at the corresponding $O$ marking is zero; however, near that marking there exist points $q_e^+$ resp. $q_e^-$ where the multiplicities of $\pi_i$ are $+1$ resp. $-1$, and the multiplicities of all other $\pi_j \ (j\neq i)$ are zero. (An example is shown in Figure~\ref{fig:periodic}.) It follows that any non-trivial linear combination of the $\pi_i$'s has some positive and some negative multiplicities. 
\end {proof}

\begin {figure}
\begin {center}
\input {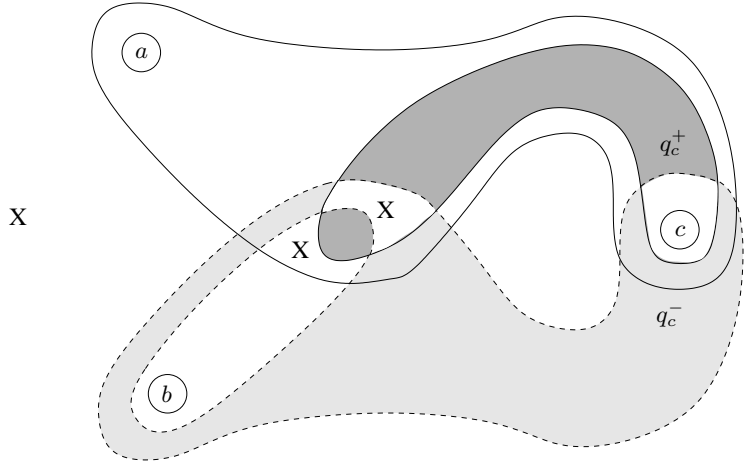}
\caption {This is a planar Heegaard diagram for an unlink $S$, obtained as follows: start with a planar projection of the unknot having a single crossing, smooth that crossing, and apply the procedure in Section~\ref{sec:planar}. In the diagram, we show the periodic domain $\pi_1 = \pi_+(c) - \pi_-(c)$ by indicating multiplicity $+1$ by darker shading, and multiplicity $-1$ by lighter shading.}
\label {fig:periodic}
\end {center}
\end {figure}

The homology of the chain complex $CFL^-(S)$ splits as
$$ HFL^-(S) = \bigoplus_{s, d \in \zz} HFL^-_d(S, s),$$
where $s$ corresponds to the Alexander grading and $d$ to the Maslov grading. Another variant of the Floer complex, $\widehat{CFL}(S)$ is gotten by choosing edges $e_{j_i}$, one  on each connected component $S_i$ ($i=0, \dots, \ell$), and setting the corresponding variables $U_{j_i}$ to zero in \eqref{eq:delhfk}. The resulting homology is denoted $\widehat{HFL}(S) = \bigoplus_{s, d} \widehat{HFL}_d(S, s).$

The Euler characteristics of $HFL^-(S)$ and $\widehat{HFL}(S)$ are related to the symmetrized Alexander polynomial of the singular link $S$. In particular, when $S = K$ is the original knot, $HFL^-(K)$ and $\widehat{HFL}(K)$ coincide with the knot Floer homologies $HFK^-(K)$ and $\widehat{HFK}(K)$, respectively, for which we have
$$ \sum_{d, s} (-1)^d \ T^s \cdot \dim \bigl( HFK^-_d(K,s) \bigr) = (1-T)^{-1} \cdot \Delta_K(T)$$
and
$$ \sum_{d, s} (-1)^d\  T^s \cdot \dim \bigl( \widehat{HFK}_d(K,s) \bigr) = \Delta_K(T),$$
where $\Delta_K(T)$ is the Alexander-Conway polynomial of $K$.

It is sometimes helpful to consider the {\em algebraic grading} $\Malg: \Ta \cap \Tb \to \zz$ on the complex $CFL^-(S)$, defined as $\Malg = M-2A$. The algebraic grading (which was denoted $N$ in \cite{CubeResolutions}) behaves like the Maslov grading in that it is decreased by one by the differential; however, it has the advantage that it is preserved by multiplication by any $U_i$.

\subsection{Insertions and connected components} 
\label {sec:insert}
Let $S$ be a partial resolution of a decorated braid projection, as in the previous subsection. Suppose we introduce a few extra two-valent vertices along the edges of the projection, which we call {\em insertions}. Let us write $S'$ for $S$ with the insertions. We then have the following variant of the planar Heegaard diagram from Subsection~\ref{sec:planar}.  Near each insertion, we introduce a new $O$ marking, a new $X$ marking, a new $\alpha$ curve, and a new $\beta$ curve, as in Figure~\ref{fig:insert}. 

\begin {figure}
\begin {center}
\input {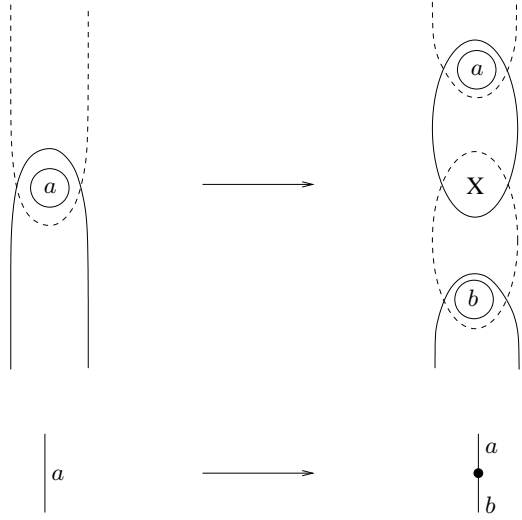}
\caption {An insertion along an edge changes the planar Heegaard diagram as shown here.}
\label {fig:insert}
\end {center}
\end {figure}

Consider the polynomial ring $\RR'$, with one $U$ variable for each $O$ marking in the new picture. 
We can define a Floer chain complex $CFL^-(S')$ over the ring $\RR'$, by the same recipe as in Subsection~\ref{sec:planar}. If $r$ is the number of insertions, the complex is constructed from tori in the symmetric product $\Sym^{2n-\sigma+r}(S^2).$ Lemma~\ref{lem:adm} easily extends to this situation.

It is worth noting that, if we view the original $CFL^-(S)$ as a complex over $\RR'$ by letting each new $U$ variable act the same way as the old variable from the edge where the insertion was done, then $CFL^-(S)$ and $CFL^-(S')$ are quasi-isomorphic over $\RR'$; see \cite[Proposition 2.3]{MOS}. 

Another useful observation is that if we consider a crossing $p$ in $\K$ that is smoothed in $S$, we can view the two resulting two-valent vertices as insertions. If we do so, the resulting Floer complex is quasi-isomorphic to the original one, in which we used the local picture in Figure~\ref{fig:exact} near $p$ (with $X$ markings at $B$). Indeed, if we handleslide $\alpha_2$ over $\alpha_1$ and $\beta_2$ over $\beta_1$ in that picture, and then do a small isotopy to separate $\alpha_1$ from $\beta_1$, we obtain exactly the Heegaard diagram for $S$ where the two-valent vertices are viewed as insertions. 

Our particular motivation for considering insertions is that, if we have a partial resolution $S$ of a decorated braid projection as in Subsection~\ref{sec:planar}, then each of its connected components can be viewed as a partial resolution of a smaller braid, with insertions. Indeed, let $S_0, \dots, S_\ell$ be the connected components of $S$, such that $S_0$ has the distinguished edge. (Here, we mean the connected components of the diagram $S$ viewed as a planar projection, rather than components of the underlying singular links. For example, the unresolved projection of a non-split link has a single component.) Any time we have a crossing in $\K$ that ends up smoothed in $S$, such that the two edges in the smoothing belong to different components $S_i$, then there are two resulting two-valent vertices, which we view as insertions. Further, for each component $S_i$ with $i > 0$, we pick one of the insertions on $S_i$ and declare it to be a gray dot as in Figure~\ref{fig:fig8} (that is, we cut the edge open at that point). With these conventions, each $S_i$ is a partial resolution of a smaller decorated braid projection, with some insertions. Thus, we can construct a Floer chain complex $CFL^-(S_i)$, over a polynomial ring $\RR_i$. We have:
$$ \RR \cong \RR_0 \otimes_\ff \cdots \otimes_\ff \RR_\ell.$$

The Floer chain complexes for the connected components are related to the original Floer complex for $S$ as follows:
 
\begin {lemma}
\label {lem:disjoint}
Let $S$ be a partial resolution of a decorated braid projection, with connected components $S_0, \dots, S_\ell$. Then, we have a quasi-isomorphism of $\RR$-modules:
$$ CFL^-(S) \sim CFL^-(S_0) \otimes_\ff  \cdots \otimes_\ff CFL^-(S_\ell) \otimes_\ff H_*(T^{\ell}).$$
Here, the right hand side is viewed as an $\RR$-module by combining the $\RR_i$-module structures on the first $\ell+1$ factors, and the $\ff$-vector space structure on the torus homology factor.  
\end {lemma}

\begin {proof}
Consider the case $\ell=1$, when the claim is that
\begin {equation}
\label {eq:disjoint}
 CFL^-(S_0 \amalg S_1) \sim CFL^-(S_0) \otimes_\ff CFL^-(S_1) \otimes_\ff H_*(S^1),
 \end {equation}

First, note that this claim is true when $S_1 = \U$ is the unknot, obtained as the braid closure of a single strand, with a gray dot and no other insertions. (An example is that in Figure~\ref{fig:periodic}, where $S_0$ is also an unknot.) Indeed, in that case the Heegaard diagram for $S_0 \amalg \U$ can be transformed by Heegaard moves so that it is obtained from the diagram for $S_0$ by adding an $\alpha$ and a $\beta$ curve, isotopic to each other and intersecting at two points, and bounding disks that contain two new basepoints (one of type $O$ and one of type $X$). An adaptation of the arguments in \cite[proof of Lemma 6.1]{Links} shows that
\begin {equation}
\label {eq:disju}
 CFL^-(S_0 \amalg \U) \sim CFL^-(S_0)[U_{\new}] \otimes_\ff H_*(S^1),
\end {equation} 
where $U_{\new}$ is the new $U$ variable corresponding to the $O$ basepoint on $\U$.

Moving to the proof of \eqref{eq:disjoint} for arbitrary $S_1$, note that the disjoint union $S_0 \amalg S_1$ can be viewed as a connected sum of  $S_0 \amalg \U$ and $S_1$, via a path connecting $S_1$ to the unknot $\U$. In \cite[Theorem 11.1]{Links}, Ozsv\'ath and Szab\'o proved a connected sum formula for link Floer complexes (for ordinary, smooth links). Their arguments extend to singular links, giving
$$ CFL^-(S_0 \amalg S_1) \sim \bigl( CFL^-(S_0 \amalg \U) \otimes_\ff CFL^-(S_1) \bigr) / (U_{\new} = U_{\old}),$$
where $U_{\old}$ is any variable corresponding to a basepoint on the component of $S_1$ joined to $\U$. Combining this with \eqref{eq:disju}, the proof is completed for $\ell=1$. The case of general $\ell$ follows by induction.
\end {proof}

\subsection {The exact triangle} 
The key ingredient in the construction of the spectral sequence from Theorem~\ref{thm:os} (a) is to establish exact triangles between the Floer complexes associated to various partial resolutions.

Precisely, consider three partial resolutions $S, \Z, \X$ that differ at a single crossing $p \in c(\K)$, where $p$ is unchanged from $\K$ in $S$  (that is, the corresponding assignment $I$ for $S$ takes $p$ to $\infty$), $\Z$ is the smoothing of $S$ at $p$, and $\X$ is the singularization. The knot Floer homologies of $S, \Z, \X$ are related by an exact triangle:

\begin {theorem}[Ozsv\'ath-Szab\'o, Corollary 4.2 in \cite{CubeResolutions}] 
\label {thm:exact}
With $S, \Z, \X$ be as above, let $a, b, c, d$ be the edges meeting at $p$ as in Figure~\ref{fig:res}, and $\L(p)$ the two-step complex
\begin {equation}
\label {eq:lp}
 \L(p) = \Bigl( \RR \xrightarrow{U_a + U_b - U_c - U_d} \RR \Bigr).
 \end {equation}
  
Then, if the crossing $p$ is positive in $S$, we have a long exact sequence:
\begin {equation}
\label {eq:tri+}
 \dots \longrightarrow HFL^-(S)  \longrightarrow H_*(CFL^-(\X) \otimes \L(p))   \longrightarrow HFL^-(\Z)  \longrightarrow \dots 
 \end {equation}

If $p$ is a negative crossing in $S$, we have a long exact sequence:
\begin {equation}
\label {eq:tri-}
 \dots \longrightarrow HFL^-(S)  \longrightarrow HFL^-(\Z)  \longrightarrow H_*(CFL^-(\X) \otimes \L(p))  \longrightarrow \dots
 \end {equation}
\end {theorem}

\begin {proof}[Sketch of proof]
Suppose $p$ is negative, so that the planar diagram for $S$ uses the basepoints $A^0$ and $A^-$ in Figure~\ref{fig:exact}. There is a subcomplex $\XX(p) \subset CFL^-(S)$ generated by those configurations that contain the point $x$ in the figure. We denote by $\YY(p)$ the associated quotient complex. Observe that $\XX(p)$ is (canonically) isomorphic to $CFL^-(\X)$, via the map that deletes $x$ from a generator. 

There is a doubly-filtered complex
 \begin {equation}
 \label {eq:square-}
 \xymatrixcolsep{5.5pc} 
\xymatrixrowsep{3.5pc}
\xymatrix{
 \XX(p) \ar[d]_{\Phi_{A^-}} \ar[r]^{\operatorname{id}} \ar[dr]^{\Phi_{A^-B}} & \XX(p) \ar[d]^{U_a + U_b - U_c - U_d} \\
 \YY(p) \ar[r]_{\Phi_B}  & \XX(p),
 }
\end {equation}
where: $\Phi_B$ is the part of the differential on $CFL^-(S)$ that counts holomorphic disks (flow lines) through exactly one of the two points marked $B$ in Figure~\ref{fig:exact} (that is, the domain of a disk should have multiplicity one at a $B$ point, and zero at the other $B$); $\Phi_{A^-}$ counts flow lines through exactly one of $A^0$ and $A^-$; and $\Phi_{A^-B}$ counts flow lines having total multiplicity one at $A^0$ and $A^-$, and also total multiplicity one at the two $B$'s. The term $U_a + U_b - U_c - U_d$ makes an appearance as the count of boundary degenerations in Figure~\ref{fig:exact} through exactly one of $A^0$ and $A^-$, and exactly one of the two $B$'s. Using the notation from the proof of Lemma~\ref{lem:adm}, the domains of these boundary degenerations are $\pi_+(a), \pi_+(b), \pi_-(c)$ and $\pi_-(d)$.

The total complex in \eqref{eq:square-} is quasi-isomorphic to its bottom row, which is $CFL^-(S)$. (The quasi-isomorphism is given by the canonical projection.) If we consider the horizontal filtration on \eqref{eq:square-}, we find a subcomplex (the right column) given by $CFL^-(\X) \otimes \L(p)$, and a quotient complex (the left column) which is $CFL^-(\Z)$. These three complexes form a short exact sequence, whose associated long exact sequence in homology is exactly \eqref{eq:tri-}.

The case when the crossing $p$ is positive is similar, but now the diagram for $S$ uses $A^0$ and $A^+$. By $\XX'(p) \subset CFL^-(S)$ we mean the subcomplex generated by configurations that contain the point $x'$ in Figure~\ref{fig:exact}. Let $\YY'(p)$ be the corresponding quotient complex, and observe that $\XX'(p)$ is still canonically isomorphic to $CFL^-(\X)$.

We have a doubly-filtered complex
 \begin {equation}
 \label {eq:square+}
 \xymatrixcolsep{5.5pc} 
\xymatrixrowsep{3.5pc}
\xymatrix{
 \XX'(p) \ar[d]_{U_a + U_b - U_c - U_d} \ar[r]^{\Phi_B} \ar[dr]^{\Phi_{A^+B}} & \YY'(p) \ar[d]^{\Phi_{A^+}} 
 \\
 \XX'(p) \ar[r]_{\operatorname{id}}  & \XX'(p),
 }
\end {equation}
where $\Phi_B, \Phi_{A^+}$ and $\Phi_{A^+ B}$ are the analogues of $\Phi_B, \Phi_{A^-}$ and $\Phi_{A^- B}$ from the negative case, but using the region marked $A^+$ instead of $A^-$. The total complex \eqref{eq:square+} is quasi-isomorphic to its top row, which is $CFL^-(S)$. The right column forms a subcomplex $CFL^-(\Z)$, and the left column a quotient complex $CFL^-(\X) \otimes \L(p).$
The associated long exact sequence in homology is \eqref{eq:tri+}.
\end {proof}

For future reference, when $S$ is a {\em complete} resolution of $\K$, we let
 \begin {equation}
 \label {eq:ls}
 \L_{S} := \bigotimes_{p \in c(S)} \L(p),
 \end {equation}
with $\L(p)$ as in \eqref{eq:lp}.

\subsection {The spectral sequence} \label {sec:spek}
As mentioned in the proof of Theorem~\ref{thm:exact}, from any crossing $p$ we can produce a  filtration on a complex quasi-isomorphic to $CFL^-(S)$, given by the horizontal direction in a diagram of the form \eqref{eq:square-} or \eqref{eq:square+}. When $S = \K$ is the original knot projection, by combining these constructions at all crossings, we can in fact build a big complex $(C_{\tot}, D_{\tot})$, which is canonically quasi-isomorphic to $CFL^-(\K)$, via contracting various identity maps that are part of $D_{\tot}$. Further, $C_{\tot}$ contains several two-step filtrations: for each crossing $p \in c(\K)$, we consider the horizontal filtration from either \eqref{eq:square-} or \eqref{eq:square+}. We let $\F$ denote the sum of all these filtrations. 

The filtration $\F$ on $C_{\tot}$ produces a spectral sequence $\{ (E_k, d_k) \}_{k \geq 0}$ that converges to $HFL^-(K)$, the homology of $H_*(C_{\tot}, D_{\tot})$. The complex $(E_0, d_0)$ (which is still $C_{\tot}$ as an $\R$-module, but with $d_0$ only made of the terms that preserve $\F$) splits as a direct sum of complexes
$$ E_0 = \bigoplus_{I: c(\K) \to \{0,1\}} C_I,$$
with each $C_I$ being quasi-isomorphic to 
$$CFL^-(S_I(\K)) \otimes \L_{S_I(\K)}.$$

This spectral sequence is exactly the one mentioned in Theorem~\ref{thm:os} (a). The $E_1$ term is described differently in the Introduction, but the two descriptions are equivalent:

\begin {proposition}[cf. Theorem 3.1 in \cite{CubeResolutions}]
\label {prop:hfs}
For any complete resolution $S$ of $\K$, we have an isomorphism
\begin {equation}
\label {eq:hfs}
 H_*(CFL^-(S) \otimes \L_S) \cong \Tor_*(\RR/L_S, \RR/N_S) \otimes \Lambda^* V_S.
 \end {equation}
\end {proposition}

\begin {proof}
Recall that $L_S \subset \R$ is the ideal generated by the elements $L(p)$ of the form $U_a + U_b - U_c - U_d$ for $p \in c(S)$. These elements  form a regular sequence in the ring $\R$; compare \cite[Lemmas 2.4 and 3.11]{RasmussenD}. Hence, the complex $\L_S$ from \eqref{eq:ls} is a Koszul resolution of $\R/L_S$. Since $\Lambda^* V_S$ is free, it follows that the right hand side of \eqref{eq:hfs} is the homology of the complex $(\RR/N_S) \otimes \Lambda^* V_S \otimes \L_S.$ Thus, it suffices to show that the complexes $CFL^-(S)$ and $\RR/N_S \otimes \Lambda^* V_S$ are quasi-isomorphic. This claim can be further reduced to the case when the complete resolution $S$ is connected, using Lemma~\ref{lem:disjoint}.

When $S$ is connected, we are left to show that $CFL^-(S)$ is quasi-isomorphic to $\RR/N_S.$ This is the content of \cite[Theorem 3.1]{CubeResolutions}. Roughly, the proof (due to Ozsv\'ath and Szab\'o, and partly based on their joint work with Stipsicz in \cite{OSSsingular}) goes as follows. They consider a different Heegaard diagram (of higher genus) for the singular knot $S$, such that the generators of the Floer complex can be related to Kauffman states for the diagram of $S$; compare \cite[Section 4]{OSSsingular}. Using this diagram they find that $HFL^-(S)$ is supported in a unique algebraic grading $\Malg$. They also consider a third diagram for $S$, which is obtained from the planar diagram from Subsection~\ref{sec:planar} by handlesliding $\alpha_2$ over $\alpha_1$ and $\beta_2$ over $\beta_1$ in Figure~\ref{fig:exact}, at all crossings $p \in c(\K)$ where $S$ is smoothed. In this third diagram, the Floer complex has a unique generator $\x$ in the lowest algebraic grading. By studying the generators in the second lowest algebraic grading, and the coefficients with which $\x$ appears in their differential, they conclude that the homology in the lowest algebraic grading is isomorphic to $\RR/N_S$.
\end {proof}

Theorem~\ref{thm:os} (a) follows directly from Proposition~\ref{prop:hfs} and the discussion preceding it. The proof of Theorem~\ref{thm:os} (b) is similar, but using Floer complexes where we set the variable $U_0$ to zero.

\begin {remark}
The statement of Theorem~\ref{thm:os} only refers to the $E_1$ pages as modules. In the original cube of resolutions with twisted coefficients from \cite{CubeResolutions}, the differential $d_1$ was also identified explicitly, in terms of zip and unzip maps. We expect that one can identify $d_1$ (and thus the $E_2$ pages) explicitly in the untwisted setting, too. However, this would require a careful analysis of the generators of $HFL^-(S)$ for disconnected resolutions $S$. (The setting with twisted coefficients was simpler because the Floer homology groups of disconnected resolutions were trivial.) 
\end {remark}

\section {HOMFLY-PT homology}
\label {sec:kr}
In this section and the following ones, we will go back to working with coefficients in $\zz$ rather than $\ff.$ We consider the base ring $\R = \zz[U_0, \dots, U_{2n}]$, as in the Introduction, and all tensor products will be taken over $\R$ unless otherwise noted.

Our main goal in this section is to prove Theorem~\ref{prop:kr}, about the HOMFLY-PT complex. 

We will work with the definition of HOMFLY-PT homology given by Rasmussen in \cite{RasmussenD} (using integral coefficients, as in Krasner's work \cite{Krasner}). It was shown in \cite[Section 3.4]{RasmussenD} that this definition is equivalent to the original one from \cite{KR2}, due to Khovanov and Rozansky. 
 
Start with a decorated braid projection $\K$ for a knot $K$, as before. We have an ideal $L_\K \subset \R$, generated by all linear elements $L(p)$ as in  \eqref{eq:Lofp}, for $p \in c(\K)$. Define the {\em edge ring}
$$ \R' := \R/L_\K.
\footnote{Our notation is the opposite of the one in \cite{RasmussenD}, where the original polynomial ring was denoted $R'$, and the edge ring was denoted $R$.}
$$

To each complete resolution $S$ of $\K$ we associate a complex $\Bkr(S)$, defined as a tensor product of $n$ two-step complexes:
$$\Bkr(S) := \bigotimes_{p \in c(K)} \Q_S(p).$$
Here, if the edges meeting at $p$ are labeled as in Figure~\ref{fig:res}, we take
\begin {equation}
\label {eq:qs}
 \Q_S(p) = \begin {cases} \R' \xrightarrow{U_a - U_c} \R' & \text{ if $S$ has a smoothing at $p$},\\
 \R' \xrightarrow{(U_a - U_c)(U_a - U_d)}\R' & \text{ if $S$ has a singularization at $p$}. 
\end {cases}
\end {equation}

The HOMFLY-PT chain complex is defined as
$$ \Ckr (\K) =  \bigoplus_{I: c(\K) \to \{0,1\}} H_*( \Bkr(S_I(\K))),$$
with a differential $D_{\KR}$ given by suitable zip and unzip maps, which can be described explicitly; see \cite{KR2} or \cite{RasmussenD} for details.

The homology of $\Ckr(\K)$ is $\Hkr(K)$, the middle HOMFLY-PT homology. If we set the variable $U_0$ to zero in the complex and then take homology, we get $\bHkr$, the reduced version of HOMFLY-PT homology. It was shown in \cite{KR2}, \cite{RasmussenD}, \cite{Krasner} that these homologies are invariants of the knot $K$.

\begin {remark}
\label {rem:mod}
Strictly speaking, this definition differs slightly from the ones in \cite{RasmussenD}, \cite{Krasner}. In \cite{RasmussenD} and \cite{Krasner} one did not have a distinguished edge subdivided in two, but rather each edge in the braid projection had its own $U$ variable. In our picture, if $U_0$ and $U_1$ are the variables corresponding to the two segments on the distinguished edge, observe that $U_0 - U_1$ is an element of $L_\K$, being equal to the sum of all linear elements $L(p)$ for $p \in c(\K)$. Thus, in the edge ring $\R'$ the variables $U_0$ and $U_1$ are identified; since our definition only involves complexes of $\R'$-modules, the end result is the same as if we had only one variable $U_0 = U_1$. 
\end {remark}

Before moving to the proof of Theorem~\ref{prop:kr}, we need a lemma. Recall that in the Introduction we defined a partial braid graph to be part of a decorated braid projection. To every connected partial braid graph $S$ we associated an ideal $Q_S$, generated by elements $Q(p)$, one for each (interior) vertex in $S$ that is not special. Here, $Q(p) = U_a U_b - U_c U_d$ if $p$ is four-valent (with outgoing edges $a$ and $b$, and incoming edges $c$ and $d$), and $ Q(p) = U_e - U_f$ if $p$ is two-valent (with outgoing edge $e$ and incoming edge $f$). We denote by $v(S)$ the set of interior (two-valent or four-valent) vertices of $S$.

\begin {lemma}
\label {lem:qregular}
Let $S$ be a connected partial braid graph, with at least two exterior edges. Then the elements $Q(p)$ for $p \in v(S)$ form a regular sequence in the ring $\R$.
\end {lemma}

\begin {proof}
We can make $\R$ into a graded ring by giving each variable $U_i$ grading one. A sequence of homogeneous elements in $\R$ is regular if and only if any permutation of the sequence is regular. Since we only consider homogeneous elements, we will not need to specify their ordering.

We will use induction on the cardinality of $v(S)$ to prove a stronger statement than the one in the Lemma, namely that:

\medskip
\noindent (*) {\it The elements $Q(p)$ for $p \in v(S)$, together with the elements $U_a$ for all incoming exterior edges $a$ of $S$, form a regular sequence $r(S)$ in $\R$.}
\medskip

The base case is when $S$ has a single vertex $p$. There are three possibilities, according to whether: $p$ is two-valent; $p$ is  four-valent and $S$ has four distinct edges meeting at $p$;  or $p$ is four-valent and $S$ has three distinct edges, one of which forming a loop from $p$ to itself. Checking (*) in each of these examples is straightforward.

For the inductive step, pick a vertex $p \in v(S)$ such that at least one of the edges coming out of $p$ is an exterior edge of $S$. Consider the partial braid graph $S'$ obtained from $S$ by deleting $p$ and the exterior edges starting or ending at $p$. By the inductive hypothesis, the claim (*) is true for $S'$. Indeed, although $S'$ may be disconnected, any of its connected components has at least one (hence at least two) exterior edges. Since the variables on different connected components are different, the sequence $r(S')$ (composed of $r(T)$ for all connected components $T$) is regular.

We now distinguish several cases:

\smallskip
\noindent (i)  \ $p$ is a two-valent vertex in $S$ with outgoing edge $e$ and incoming edge $f$.  Then the sequence $r(S)$ is obtained from the regular sequence $r(S')$  by adding $U_e - U_f$. Since the variable $U_e$ did not appear in $r(S')$, the new sequence $r(S)$ must be regular. 

\smallskip
\noindent (ii) \ $p$ is a four-valent vertex in $S$ with both outgoing edges $a$ and $b$ being exterior. Let $c$ and $d$ be the incoming edges at $p$. Note that if $S$ is connected and has more than one crossing, it cannot be that both $c$ and $d$ are exterior edges. Therefore, we have two subcases:
\begin {itemize}
\item Neither of the incoming edges is exterior in $S$. Then $r(S)$ is obtained from $r(S')$ by adding $U_a U_b - U_c U_d$. Since $U_a$ and $U_b$ do not appear in $r(S')$, we get that $r(S)$ is regular.

\item One of the incoming edges (say, $c$) is exterior in $S$. Then $r(S)$ is obtained from $r(S')$ by adding $U_c$ and $U_a U_b - U_c U_d$. Since $U_a, U_b$ and $U_c$ do not appear in $r(S')$, again we get that $r(S)$ is regular.
\end {itemize}

\smallskip
\noindent (iii) \  $p$ is a four-valent vertex in $S$ with only one outgoing edge being exterior. Say that $a$ is the exterior outgoing edge, $b$ the other outgoing edge at $p$, and $c$ and $d$ the incoming edges. Let $I$ be the ideal of $\R$ generated by all elements of $r(S')$ except $U_b$. We have three subcases:
\begin {itemize}
\item Neither of the incoming edges is exterior in $S$. Then $r(S)$ is obtained from $r(S')$ by deleting $U_b$ and adding $U_a U_b - U_c U_d$. We know that $U_b$ is not a zero-divisor in $\R/I.$ The same must be true for $U_a U_b - U_c U_d$, because the variable $U_a$ does not appear in $r(S')$. Therefore, $r(S)$ is regular.

\item Exactly one of the incoming edges (say, $c$) is exterior in $S$. Then $r(S)$ is obtained from $r(S')$ by deleting $U_b$, and adding $U_c$ and $U_a U_b - U_c U_d$. Again, we know that $U_b$ is not a zero-divisor in $\R/I$. Since $U_c$ and $U_a$ do not appear in $r(S')$, we get that $U_b$ is not a zero-divisor in $\R/(I + (U_c))$, and from here that the new sequence $r(S)$ is regular. 

\item Both $c$ and $d$ are exterior edges. Then $r(S)$ is obtained from $r(S')$ by deleting $U_b$, and adding $U_c, U_d,$ and $U_a U_b - U_c U_d$. Since $U_b$ is not a zero-divisor in $\R/I$, and $U_c, U_d, U_a$ do not appear in $r(S')$, it follows that $r(S)$ is regular. 
\end {itemize}

This completes the inductive proof of (*). \end {proof}

\begin {proof}[Proof of Theorem~\ref{prop:kr}] If $x_1, \dots, x_n$ are (not necessarily distinct) elements of $\R$, we will denote by $\Koszul\{x_1, \dots, x_n\}$ the Koszul complex associated to $x_1, \dots, x_n$, that is, 
$$ \Koszul\{x_1, \dots, x_n\} = \bigotimes_{i=1}^n \bigl( \R \xrightarrow{ x_i } \R \bigr).$$

Moreover, adjusting the notation in \eqref{eq:lp} and \eqref{eq:ls} to the coefficient ring $\R$ rather than $\RR = \R \otimes_{\zz} \ff$, we set
$$ \L(p) = \Koszul\{L(p)\}, \ p \in c(\K),$$
and
$$ \L_S = \bigotimes_{p \in c(S)} \L(p) = \Koszul \{ L(p) \mid p \in c(S) \}.$$
Recall from the proof of Proposition~\ref{prop:hfs} that the elements $L(p), p \in c(S)$ (which generate the ideal $\L_S \subset \R$) form a regular sequence. Thus, $\L_S$ is a free resolution of the quotient module $\R/L_S.$ In particular, $\L_\K$ is a free resolution of the edge ring $\R'$, viewed as an $\R$-module.

Let $p \in c(\K)$. If $S$ has a smoothing at $p$ (that is, $p \in c(\K) \setminus c(S)$), then $p$ produces two vertices in the graph of $S$, which we denote by $p^l$ (the one on the left) and $p^r$ (the one on the right). Equation \eqref{eq:qs} then reads $\Q_S(p) = \R' \otimes \Koszul\{Q(p^l)\}.$ If $S$ has a singularization at $p$, since $U_a + U_b - U_c - U_d = 0$ in $\R'=\R/L_{\K}$, we get $-(U_a-U_c)(U_a - U_d) = U_a U_b - U_c U_d,$ so Equation \eqref{eq:qs} can be read as $\Q_S(p) = \R' \otimes \Koszul\{Q(p)\}.$ Thus,
$$ \Bkr(S) \cong \R' \otimes \Koszul\bigl( \{ Q(p) \mid p \in c(S) \} \cup \{ Q(p^l) \mid p \in c(\K)\setminus c(S)\} \bigr) .$$

Note that the ideal $L_\K$ differs from $L_S$ in that the generator set of $L_\K$ also contains the elements $L(p)$, where $p$ is a crossing of $\K$ smoothed in $S$. Since these elements form a regular sequence, $\R' = \R/L_\K$ is quasi-isomorphic (over $\R$) to
$$\R/L_S \otimes \Koszul\{L(p) \mid p \in c(\K) \setminus c(S) \}.$$

From here we get a quasi-isomorphism
\begin {equation}
\label {eq:bucurs}
 \Bkr(S) \sim \R/L_S \otimes \Koszul\bigl(  \{ Q(p) \mid p \in c(S) \} \cup \{ L(p), Q(p^l) \mid p \in c(\K)\setminus c(S)\} \bigr).
 \end {equation}

Since $L(p) = Q(p^l) + Q(p^r)$ for $p \in c(\K)\setminus c(S)$, it follows that (up to quasi-isomorphism)
we can replace $L(p)$ with $Q(p^r)$ in \eqref{eq:bucurs}. Recall that $v(S) = c(S) \cup \{p^l, p^r \mid p \in c(\K)\setminus c(S)\}$ is the set of (interior) vertices in the graph of $S$. Therefore, we can write
\begin {equation}
\label {eq:bucurs2}
 \Bkr(S) \sim \R/L_S \otimes \Koszul \{ Q(p) \mid p \in v(S) \}.
  \end {equation}

Let us denote by $e_0, e_1, \dots, e_k$ the edges in $\K$ that are drawn around the braid to take its closure, ordered from right to left (on the right side of the diagram), as in Figure~\ref{fig:fig8}. Each connected component of the complete resolution $S$ contains a certain number of consecutive edges among $e_0, \dots, e_k$. If $S$ has $m+1$ connected components, we denote them by $S_0, \dots, S_m$, so that $S_j$ contains $e_{i_j}, \dots, e_{i_{j+1} -1},$ for $0 = i_0  < i_1 < \dots < i_{j+1} = k+1.$

For each $j=1, \dots, m,$ let $p_j \in v(S)$ be a two-valent vertex on $S_j$ coming from the right hand side of a smoothed crossing in $\K$ whose left hand side ended up in $S_{j-1}$. (In the terminology from the Introduction, $p_j$ is the {\em special} vertex on $S_j$.) Let $S_j'$ denote the partial braid graph obtained from $S_j$ by deleting the vertex $p_j$, so that the two edges meeting at $p_j$ become exterior edges. (In the particular case when $S_j$ has a single edge from $p_j$ to itself, we let $S'_j = \emptyset.$) We have $v(S_j) = v(S'_j) \cup \{p_j\}.$ Set also $S'_0 = S_0$. Thus, each $S_i'$ is a connected partial braid graph with at least two exterior edges.

For $j=1, \dots, m,$ since $S_j$ has no exterior edges, the sum of the linear elements $L(p)$ for $p \in c(S_j)$ and $Q(p)$ for two-valent vertices $p \in v(S_j)$ is exactly zero. We get that the sum of $Q(p)$ for two-valent vertices $p \in v(S_j)$ is zero in $\R/L_S.$ By taking linear combinations of the generators in a Koszul complex, we can transform  \eqref{eq:bucurs2} into
\begin {equation}
\label {eq:bucurs3}
 \Bkr(S) \sim \R/L_S \otimes \Lambda^*V_S \otimes \Koszul  \{ Q(p) \mid p \in v(S), p \neq p_j \text{ for any } j > 0 \}.
 \end {equation}
Indeed, after tensoring with $\R/L_S$, each component $S_j$ ($j > 0$) produces a term $\Koszul\{0\},$ and together these terms give $\Lambda^* V_S$. 

Note that
$$  \{ Q(p) \mid p \in v(S), p \neq p_j \text{ for any } j \} = \bigcup_{j=0}^m \{Q(p) \mid p \in v(S_j') \},$$
and, by definition, this set generates the ideal $Q_S \subset \R$.

By Lemma~\ref{lem:qregular}, the elements $Q(p)$ for $p \in v(S_j')$ (with $j$ fixed) form a regular sequence. Since the variables along the edges of $S_j'$ are different for different $j$, it follows that all the elements $Q(p)$ that produce the Koszul complex in \eqref{eq:bucurs3} form a regular sequence in $\R$. Hence, that Koszul complex is a free resolution of the module $\R/Q_S$. Tensoring this resolution with $\R/L_S$ and then taking homology we obtain $\Tor_*(\R/L_S, \R/Q_S)$. From \eqref{eq:bucurs3} we see that $H_*(\Bkr(S)) \cong \Tor_*(\R/L_S, \R/Q_S) \otimes \Lambda^* V_S$, as desired.
\end {proof}

\section {Gradings}
\label {sec:gr}
The HOMFLY-PT chain complex and its homology are triply graded---see \cite{KR2}, \cite{RasmussenD}. Conjecture~\ref{conj1} relates the HOMFLY-PT complex $\Ckr(\K)$ to the complex $\C(\K)$ from the Introduction, which gives the $E_1$ page of the spectral sequence in Theorem~\ref{thm:os}. Thus, we expect $\C(\K)$ to have three gradings as well. In this section we construct these three gradings on $\C(\K)$, which we denote by $\tgr_q, \tgr_h,$ and $\tgr_v$. We will then state a graded refinement of Conjecture~\ref{conj1}.

We first define gradings $\gr_q$ and $\gr_h$ on the ring $\R$ by setting $\gr_q(U_i) = 2$ and $\gr_h(U_i) = 0$ for each variable $U_i$. (In particular, $\gr_q$ is twice the grading on $\R$ considered in Lemma~\ref{lem:qregular}.) Let $S$ be a complete resolution of $\K$ and $p \in c(S)$. We extend $\gr_q$ and $\gr_h$ to gradings on the mapping cone 
$$\L(p) = \Bigl( \R \xrightarrow{U_a + U_b - U_c - U_d} \R \Bigr) $$
by shifting the gradings of the first $\R$ term upwards by $2$ in $\gr_q$, and downwards by $1$ in $\gr_h$. (This way, the map defining $\L(p)$ preserves $\gr_q$ and increases $\gr_h$ by one.) Next, we equip $\R/N_S$ with the bigrading descended from $\R$. We also assign bigrading $(\gr_q, \gr_h) = (2, -1)$ to each generator of $V_S$, and this induces a bigrading on the wedge product $\Lambda^* V_S$. 

Define the complex
$$\B(S) := \L_S \otimes (\R/N_S) \otimes \Lambda^*V_S,$$
whose homology is $\hB(S) = \Tor_*(\R/L_S, \R/N_S) \otimes \Lambda^*V_S$; compare the proof of Proposition~\ref{prop:hfs}. 

By construction, we have a bigrading $(\gr_q, \gr_h)$ on $\B(S)$ and on its homology. Note that on the homology $\hB(S) = \Tor_*(\R/L_S, \R/N_S) \otimes \Lambda^*V_S$, we can get $\gr_h$ alternatively as minus the sum of the natural gradings on $\Tor_*$ and $\Lambda^*$.

From here we get a bigrading on the group $\C(\K) = \bigoplus_I \hB(S_I(\K))$ by  normalizing $\gr_q$ and $\gr_h$ as follows. Let $k$ be the braid index of $\K$, and $N_+, N_-$ be the number of positive resp. negative crossings in $\K$. For $I: c(\K) \to \{0,1\},$ we let $\|I\| = \sum_{p \in c(\K)} I(p)$. On a term $\hB(S) \subseteq \C(\K)$ with $S = S_I(\K)$, we set
\begin {eqnarray*}
 \tgr_q &=& \gr_q -  \ \# c(S) - \| I \| + N_- + k, \\
  \tgr_h &=& \gr_h + (N_+ - N_- + k - 1)/2. 
  \end {eqnarray*}

We also define a third grading on $\C(\K)$ that (up to a constant) measures the depth in the hypercube of resolutions:  
$$ \tgr_v = \| I \| - (N_+ + N_- + k - 1)/2.$$

It is instructive to relate our gradings to the usual ones for knot Floer homology, from \cite{Knots}, \cite{OSSsingular}, \cite{CubeResolutions}; see also Subsection~\ref{sec:planar}. On the complex $\B(S)$ and its homology $\hB(S)$ we define {\em Alexander} and {\em Maslov} gradings by
$$ A = (-\gr_q +\  \# c(S) - k+1) /2, \ \  M = 2A - \gr_h. $$

We equip the complex $\C(\K)$ with a Maslov grading $M$ coming from the one on each $\hB(S)$, and 
to a  normalized Alexander grading given by:
$$ A' = A + (\| I \| - N_-)/2 = (- \tgr_q + 1)/2. $$

Observe that the Maslov grading on $\C(\K)$ can also be written as
$$ M = - \tgr_q - \tgr_h - \tgr_v + 1.$$

These definitions coincide with the ones used by Ozsv\'ath and Szab\'o in \cite{CubeResolutions}. Note that $-\gr_h$ corresponds to the algebraic grading $N=2A-M$ from \cite[Section 2.3]{CubeResolutions}. Indeed, we can see that $N$ is the same as our $-\gr_h$ as follows: By the arguments in \cite[proof of Theorem 3.1]{CubeResolutions}, we have that $N=-\gr_h$ for the bottom degree generator of $\B(S)$; the general identification is then obtained by keeping track of the gradings in the proof of Proposition~\ref{prop:hfs}.

Recall from Subsection~\ref{sec:spek} that the spectral sequence from Theorem~\ref{thm:os} is induced by a filtration $\F$ on a complex $C_{\tot}.$ As a group, $C_{\tot} = \C(\K)$ splits as 
$$\bigoplus_{I:c(\K) \to \{0,1\}} C_I,$$
with each $C_I$ in filtration degree $-\| I \|$. It is proved in \cite[Section 4.1]{CubeResolutions} that the total differential on $C_{\tot}$ preserves $A'$ and decreases $M$ by one. Moreover, by construction, the differential $d_\ell$ on the $E_\ell$ page of the spectral sequence must increase $\tgr_v$ by $\ell$. In all, it follows that the $d_\ell$ changes the  triple grading $(\tgr_q, \tgr_h, \tgr_v)$ by $(0, 1-\ell, \ell).$ In particular, the differential $d_1$ on $\C(\K)$ preserves $\tgr_q$ and $\tgr_h$ and increases $\tgr_v$ by one.

Therefore, the group $\C(\K)$ splits as
$$ \C(\K) = \bigoplus_{i, j, k \in \zz} \C^{i, j, k}(\K),$$
where we let $x \in \C^{i, j, k}(\K)$ if $x$ is homogeneous with respect to the three gradings, and 
$(i, j, k) = (\tgr_q(x), 2\tgr_h(x), 2\tgr_v(x)).$ We let $H^{i, j, k}(\K)$ be the homology of $\C(\K)$ in the given triple grading, with respect to the differential $d_1$. 

We also obtain induced gradings on the complex $\C(\K)/(U_0 = 0)$ and its homology. We denote a triply graded piece of the homology of $\C(\K)/(U_0 = 0)$ by  $\bH^{i, j, k}(\K),$ but here $(i, j, k) = (\tgr_q(x) -1, 2\tgr_h(x), 2\tgr_v(x)).$

We have chosen our notation to be parallel to that in \cite{RasmussenD}, where Rasmussen defined three gradings $q, \gr_h$ and $\gr_v$ (where $q$ corresponds to our $\gr_q$) on the complex $\Ckr(\K)$, in a very similar way. (See \cite{RasmussenD} for more details.) He then normalized the gradings to get splittings of the middle and reduced HOMFLY-PT homologies
$$ \Hkr(K) = \bigoplus_{i, j, k \in \zz} \Hkr^{i, j, k}(K)  \ \ \text{ and }\ \  \bHkr(\K) = \bigoplus_{i, j, k \in \zz} \bHkr^{i, j, k}(K).$$

The (bigraded) Euler characteristics of these homologies are
\begin {eqnarray*}
 \sum_{i, j, k} (-1)^{(k-j)/2} a^j q^i \rk \bigl( \Hkr^{i, j, k} (K) \bigr) &=& P_K(a, q)/(q^{-1} - q),\\
 \sum_{i, j, k} (-1)^{(k-j)/2} a^j q^i \rk \bigl( \bHkr^{i, j, k} (K) \bigr) &=& P_K(a, q),
 \end {eqnarray*}
where $P_K(a, q)$ is the HOMFLY-PT polynomial of $K$, normalized to be $1$ on the unknot and to satisfy the skein relation:
\begin{equation*}
a P_{\ \includegraphics[scale=0.2]{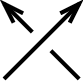}}(a, q) - a^{-1} P_{\ \includegraphics[scale=0.2]{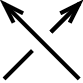}}(a, q) = (q-q^{-1}) P_{\ \includegraphics[scale=0.2]{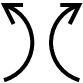}}(a, q).
\end{equation*} 

\begin {remark}
The specialization $\Delta_K(T) = P_K(1, T^{1/2})$ gives the Alexander-Conway polynomial of $K$, mentioned in Subsection~\ref{sec:planar}.
\end {remark}

We are now able to state the following strengthened version of Conjecture~\ref{conj1}: 
\begin {conjecture}
\label {conj1s}
Let $K \subset S^3$ be an oriented knot, with a decorated braid projection $\K$. For any $i, j, k \in \zz,$ we have isomorphisms
$$ H^{i, j, k}(\K) \cong \Hkr^{i, j, k}(K) \ \text{ and } \ \bH^{i, j, k}(\K) \cong \bHkr^{i, j, k}(K).$$
\end {conjecture}

\begin {remark}
\label {rem:lit}
For the readers more familiar with other sources, it is worth recalling how Rasmussen's conventions compare with others. In the original reference \cite{KR2}, Khovanov and Rozansky had three gradings  as well. As mentioned in \cite[Proposition 3.13]{RasmussenD}, an element with grading $(i, j, k)$ in Rasmussen's notation corresponds to one with grading $(j/2, i-j/2, k/2)$ in the notation of \cite{KR2}. Also, in \cite{DGR}, Dunfield, Gukov, and Rasmussen worked with a polynomial in three variables $a, q, t$. A homology generator in grading $(i, j, k)$ in the notation of \cite{RasmussenD} corresponds to a monomial $a^j q^i t^{(j-k)/2}$ in the notation of \cite{DGR}.
\end {remark}

We saw that the complex $\C(\K)$ admits a triple grading $(i, j, k) = (\tgr_q, 2\tgr_h, 2\tgr_v)$. We also saw that the differential $d_\ell$ on the $E_\ell$ page of the spectral sequence from Theorem~\ref{thm:os} changes this  triple grading by $(0, 2-2\ell, 2\ell).$ In particular, when $\ell=2$ the grading change is by $(0, -2, 4),$ which translates into $(-2, 0, -3)$ in the conventions of \cite{DGR}; see Remark~\ref{rem:lit}. This exactly corresponds to the projected behavior of the ``$d_0$ differential'' in \cite{DGR}. Thus, if Conjectures~\ref{conj1} and \ref{conj1s} were true and the spectral sequence happened to collapse at the $E_2$ stage, Theorem~\ref{thm:os} would imply that knot Floer homology can be obtained from HOMFLY-PT homology by introducing a differential with the grading properties predicted by Dunfield, Gukov, and Rasmussen in \cite{DGR}.

\section {Partial braid graphs and Tor groups}
\label {sec:partial}

This section contains a discussion of Conjecture~\ref{conj2}, about partial braid graphs. In the Introduction, partial braid graphs were defined as subsets of decorated braid projections, where the distinguished edge of the braid projection is viewed as split open into two edges. Alternately, we can give a more intrinsic definition (equivalent to the previous one), as follows.

An {\em open partial braid graph} $\tg$ consists of a finite collection of smooth arcs $\gamma_1, \dots, \gamma_n: [0,1] \to D = [0,1] \times [0,1]$, and a finite collection of vertices $W = \{p_1, \dots, p_m\}$, with the following properties:
\begin {itemize}
\item For each $i,$ the second coordinate $\gamma_i^{(2)}$ of the arc $\gamma_i = (\gamma_i^{(1)}, \gamma_i^{(2)})$ satisfies $\bigl( \gamma_i^{(2)}\bigr)'(t) > 0$ for all $t \in [0,1];$
\item Each $p_j \in W$ lies in the interior of one (or two) arcs $\gamma_i$;
\item Any two arcs intersect transversely, and only in their interior; every intersection point of two arcs is one of the vertices in $W$;
\item The intersection of any three arcs is empty;
\item The number $k$ of arcs with the initial point on $[0,1] \times \{0\}$ is the same as the number of arcs with the final point on $[0,1]\times \{1\}$.
\end {itemize}

An open partial braid graph can be thought of as a particular kind of oriented graph with only univalent, two-valent and four-valent vertices. The univalent vertices (not part of $W$) are the ends of the arcs $\gamma_i$. An example of an open partial braid graph is shown in Figure~\ref{fig:braid}. 

A {\em partial braid graph} $S=\hat \tg$ is defined to be the braid closure of an open partial braid graph $\tg$. This braid closure is obtained by joining the univalent vertices on $[0,1] \times \{0\}$ with the univalent vertices on $[0,1] \times \{1\}$ using $k$ strands on the right, as in Figure~\ref{fig:braid}. We then erase the univalent vertices that were joined by strands. Thus, $\hat \Gamma$ has $2k$ fewer univalent vertices than $\Gamma.$ The univalent vertices of $\hat \Gamma$ are called {\em loose ends}, and the four-valent vertices are called {\em crossings}. The two-valent vertices do not play an essential role, and we will mostly focus on partial braid graphs without two-valent vertices; see Subsection~\ref{sec:no2} below for the relevant discussion. Also, for convenience, we will only discuss {\em connected} partial braid graphs. 

\begin {figure}
\begin {center}
\input {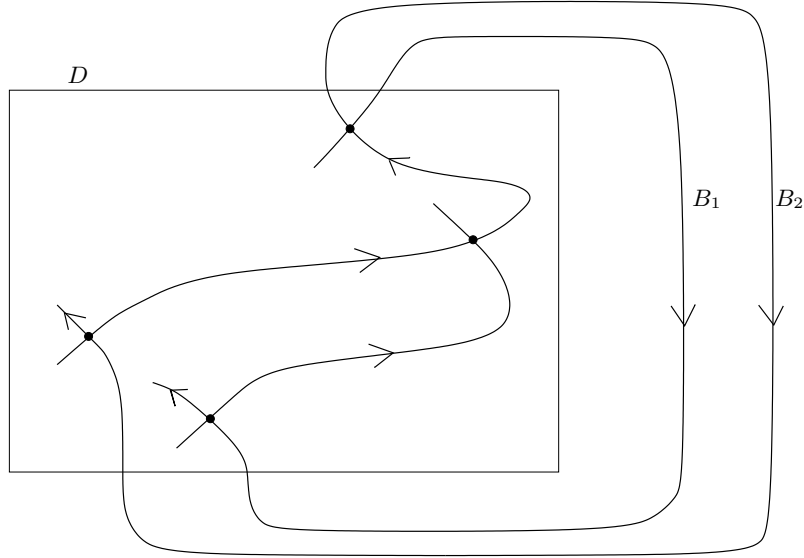}
\caption {A partial braid graph and its closure. Here $n=5$ and $k=2$. The closure $\hat \Gamma$ has six loose ends, four crossings, and no two-valent vertices.}
\label {fig:braid}
\end {center}
\end {figure}

Moreover, we impose another assumption on partial braid graphs:

\begin {assumption}
\label {as}
 $S = \hat \Gamma$ must contain at least one (hence at least two) loose ends. 
 \end {assumption}
 
 This condition is necessary for Conjecture~\ref{conj2} to have a chance of being true; see Subsection~\ref{sec:noas} below for an explanation. From now on, we will always assume that the partial braid graphs are connected and satisfy Assumption \ref{as}.

Let $S = \hat \tg$ be a partial braid graph, with $W$ being the set of its (two-valent and four-valent) vertices. We let $c(S) \subseteq W$ be the subset of four-valent vertices. We also let $E$ be the set of edges of $S$. Each edge $e \in E$ has an induced orientation,  and an initial and a final point; these can be either loose ends, or vertices in $W$.

For each edge $e \in E,$ we introduce a variable $U_e.$ We consider the ring 
$$\R = \zz[\{U_e | e \in E\}].$$

Starting from here, we can define the ideals $L=L_S, N=N_S$ and $Q=Q_S$ from the Introduction, intrinsically in terms of $S$. The ideal $L$ is generated by linear elements $L(p) \in \R$, one for each four-valent vertex $p \in c(S).$ The ideal $Q$ is generated by elements $Q(p), p \in W,$ which are quadratic for four-valent vertices, and linear for two-valent vertices. The ideal $N$ is generated by homogeneous elements $N(W')$, one for each subset  $W' \subseteq W.$ Conjecture~\ref{conj2} claims the existence of $\R$-module isomorphisms
\begin {equation}
\label {eq:tors}
 \Tor_i(\R/L, \R/N) \cong \Tor_i(\R/L, \R/Q),
 \end {equation}
 for all $i \geq 0$.

\medskip 

The rest of this section is devoted to various remarks about Conjecture~\ref{conj2}.

\subsection{Gradings}
\label {sec:conjgr}
Recall that in Section~\ref{sec:gr} we equipped the complex $\C(K)$ with three gradings $\gr_q, \gr_h$ and $\gr_v$, similar to the well-known ones on the HOMFLY-PT complex. It is natural to expect that there is a graded version of Conjecture~\ref{conj2} consistent with the statement of Conjecture~\ref{conj1s}, so that the gradings can be identified at the level of all partial braid graphs. Indeed, the grading $i$ in \eqref{eq:tors} corresponds to $\gr_h$. On the other hand, the grading $\gr_v$ has to do with the relative position in the cube of  resolutions, so it is not visible when we talk about partial braid graphs intrinsically. 

There is still the grading $\gr_q$. For partial braid graphs, we can define $\gr_q$ on $ \Tor_i(\R/L, \R/N)$ by the same rules as in Section~\ref{sec:gr}: each variable $U_i$ is set in grading level $2$, inducing a grading on $\R$ and $\R/N$; then we compute the $\Tor$ group as the homology of the complex $\R/N \otimes \L$, where in 
$$  \L :=  \bigotimes_{p \in c(S)} \Bigl ( \R \xrightarrow{L(p)} \R \Bigr), $$
we shift the grading of the first $\R$ term in each parenthesis upward by $2$ (so that the differential of  the Koszul complex $\L$ preserves the grading $\gr_q$).

Let us define $\gr_q$ on $\Tor_i(\R/L, \R/Q)$ in the same way. However, this does not exactly correspond to the $q$-grading on the HOMFLY-PT complex, as defined in \cite{KR2} or \cite{RasmussenD}, because there the differential $d_+$ at each vertex increases $\gr_q$ by $2$ (instead of preserving it). Thus, we must be careful when relating the gradings $\gr_q$ on the two sides of \eqref{eq:tors}. We arrive at the following graded version of Conjecture~\ref{conj2}, which is the one consistent with Conjecture~\ref{conj1s}, and with our computations:

\begin {conjecture}
\label {conj3}
Let $S$ be a connected partial braid graph (satisfying Assumption~\ref{as}). Then there exist isomorphisms \eqref{eq:tors}, such that the elements in $\gr_q$-grading level $j$ on the left hand side correspond to elements in $\gr_q$-grading level $j+2i$ on the right hand side.
\end {conjecture}

\subsection {Failure of the obvious maps}
\label{sec:fail}
Note that $Q \subseteq N,$ since every $Q(p)$ equals either $N(\{p\})$ (in case there is no loop in $E$ from $p$ to itself), or $N(\{ p \}) U_e,$ if there is such a loop $e.$ Hence, there is a natural quotient map $\R/Q \to \R/N$ which induces natural maps
\begin {equation}
\label {eq:fi}
 f_i :   \Tor_i(\R/L, \R/Q) \to \Tor_i(\R/L, \R/N).
\end {equation}

However, in general the maps $f_i$ are not the desired isomorphisms from \eqref{eq:tors}. Indeed, this would not be consistent with the proposed grading identification from Conjecture~\ref{conj3}. More concretely, as an example, consider the partial braid graph from Figure~\ref{fig:ox}. Then:
$$ \R = \zz[U_1, U_2, U_3], \ L = N = (U_1 - U_2), \ Q = (U_1 U_3 - U_2 U_3).$$

\begin {figure}
\begin {center}
\input {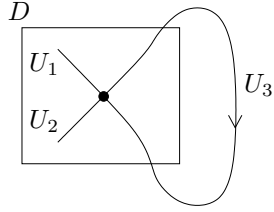}
\caption {A partial braid graph with one crossing. We write the corresponding $U$ variable on each edge.}
\label {fig:ox}
\end {center}
\end {figure}

\noindent Both $\Tor_1(\R/L, \R/Q)$ and $\Tor_1(\R/L, \R/N)$ are isomorphic to $\R/(U_1 - U_2),$ as can be seen by tensoring the Koszul resolution 
$$\L(p) = \Bigl( \R \xrightarrow{U_1 - U_2} \R \Bigr)$$ with $\R/Q$ resp. $\R/N$, and then taking homology.  However, under the natural isomorphisms of the $\Tor$ groups with $\R/(U_1 - U_2),$ the map $f_1$ corresponds to multiplication by $\pm U_3$, which is not an isomorphism.

\subsection {Two-valent vertices}
\label {sec:no2}
Let $S$ be a partial braid graph, and $S'$ be the graph obtained from $S$ by inserting a new two-valent vertex $p$ on an edge $a$. (Compare Subsection~\ref{sec:insert}. Going from $S'$ to $S$ is the operation of mark removal, discussed in \cite[Lemma 3]{KR2} and \cite[Section 2.2]{RasmussenD}.) In $S'$, we keep the notation $a$ for the outgoing edge from $p$, and we let $b$ the incoming edge at $p$, as in Figure~\ref{fig:insert}. 

The base ring $\R'$ for $S'$ contains the variables $U_a$ and $U_b$. It is related to the base ring $\R$ for $S$ by the relation 
$$ \R' = \R/(U_a - U_b).$$

We denote by $L', N', Q'$ the ideals in $\R'$ analogous to $L, N, Q$ in $\R$.

\begin {lemma}
\label {lem:torres}
If $S'$ is obtained from $S$ by inserting a two-valent vertex as above, then:

$(a)$ We have isomorphisms of $\R'$-modules
$$ \Tor^{\R}_i (\R/L, \R/N) \cong \Tor^{\R'}_i (\R'/ L', \R'/N')$$
and
$$ \Tor^{\R}_i (\R/L, \R/Q) \cong \Tor^{\R'}_i (\R'/ L', \R'/Q').$$
Here,  the superscripts $\R$ and $\R'$ indicate the base ring for the $\Tor$ groups, and an $\R$-module is viewed as an $\R'$-module with $U_a$ and $U_b$ acting the same way.

$(b)$ The statement $L+Q = L+N$ is equivalent to the statement $L'+Q' = L'+N'.$
\end {lemma}

\begin {proof}
(a) Note that $U_a - U_b \in Q' \subseteq N'$, and that under the projection $\R' \to \R,$ the ideals $L', N', Q'$ project to the corresponding ideals $L, N, Q$. We think of each $\Tor$ group as the homology of a complex obtained from a free resolution of $\R/L$ (or $\R'/L'$), by tensoring with a second module. The claimed isomorphisms on homology follow from corresponding isomorphisms at the level of these complexes.

(b) $L'+Q'=L'+N'$ implies the other statement using the projection $\R' \to \R.$ For the converse, suppose $L+Q = L+N$. Since $U_a - U_b \in L'+Q'$, we see that $L'+Q'$ is generated by the same 
elements as $L+Q$, together with $U_a - U_b$. Similarly, $L'+N'$ is generated by the same elements as $L+N$, together with $U_a - U_b$. Therefore, $L'+Q' = L'+N'.$ 
\end {proof}

In light of Lemma~\ref{lem:torres} (a), Conjecture~\ref{conj2} can be reduced to the case where there are no two-valent vertices.

\subsection {Vanishing results}
\label {sec:acyclic}

A case in which Conjecture~\ref{conj2} is easy to prove is when the open partial braid graph $\Gamma$ does not intersect the top and bottom edges of the rectangle $D$; that is, taking its braid closure is a vacuous operation, and $\hat \Gamma = \Gamma.$ We have:

\begin {proposition}
Suppose $\hat \Gamma = \Gamma.$ Then:

(a) The ideals $N$ and $Q$ coincide, so $\Tor_i(\R/L, \R/N) = \Tor_i(\R/L, \R/Q)$ for all $i$.

(b) In fact, $ \Tor_i(\R/L, \R/Q) = 0 \ \text{ for } i > 0.$
\end {proposition}

\begin {proof}
(a) In this situation all the generators $N(W') \in N$  are in the ideal $Q$; compare \cite[Lemma 3.12]{CubeResolutions}. This can be proved by induction on the number of elements in $W'$: For the inductive step, notice that if we let $p$ be the topmost vertex in $W'$, then $N(W')$ is in the ideal $(Q(p)) + N(W' \setminus \{p\})$.

(b) This follows from the fact that the generators $L(p), Q(p)$ of $L$ and $Q$ form a regular sequence in $\R$; see \cite[Lemma 1]{KhSoergel} for the proof.
\end {proof}

A related result is the following:

\begin {lemma}
\label {lem:bound}
Let $\Gamma$ be any open partial braid graph, with a connected braid closure $\hat \Gamma$  obtained by closing up $k$ strands. Then
$$\Tor_i(\R/L, \R/Q) = 0 \ \text{ for } i > k.$$
\end {lemma}

\begin {proof}
By Lemma~\ref{lem:torres} (a), we can assume without loss of generality that each of the strands used to take the braid closure has a two-valent vertex just before the top edge $[0,1] \times \{1\}$ in $D$. Let $a_j$ (resp. $b_j$) the outgoing (resp. incoming) edge at these two-valent vertices, for $j=1, \dots, k$.

The proof of Theorem~\ref{prop:kr} from Section~\ref{sec:kr} extends to (connected) partial braid graphs, implying that $\Tor_*(\R/L, \R/Q)$ is the homology of the Koszul complex 
\begin {equation}
\label {eq:koz}
\Koszul \bigl( \{L(p) \mid  p \in c(\hat \Gamma)  \} \cup \{Q(p) \mid p \in W  \} \bigr ),
\end {equation}
in the notation of Section~\ref{sec:kr}. Among the generators $Q(p)$ we find $U_{a_j} - U_{b_j}, j=1, \dots, k$. If we eliminate these, the rest of the $Q(p)$'s are the generators of the quadratic ideal $Q'$ for the open partial braid $\Gamma'$, obtained from $\Gamma$ by removing the $k$ two-valent vertices at the top. Also, the generators $L(p)$ for the linear ideal $L$ are the same as those for the similar ideal $L'$ for $\Gamma'$. By  \cite[Lemma 1]{KhSoergel}, the generators of $Q'$ and $L'$ form a regular sequence. Therefore, the Koszul complex \eqref{eq:koz} is quasi-isomorphic to
$$ \Koszul(\{ U_{a_j} - U_{b_j} \mid j =1, \dots, k \} ) \otimes \R/(Q'+L').$$
This complex is only supported in degrees up to $k$, hence so is its homology.
\end {proof}

\subsection {Proof of Theorem~\ref{thm:2}}
Let $S$ be any partial braid graph. We want to show that the ideals $L+Q , L+N \subseteq \R$ are the same, so that when $i=0$ the map $f_0$ from \eqref{eq:fi} is the desired isomorphism in Conjecture~\ref{conj2}. The fact that $L+Q = L+N$ will follow readily from Proposition~\ref{main} below. Indeed, given a subset $W' \subseteq W$, let $S'$ be the partial braid graph consisting of all the vertices in $W'$, together with all the edges in $\Outof(W') \cup \Into(W')$. Applying Proposition~\ref{main} to $S'$ (or, if $S'$ is disconnected, to its connected components), we get that $N(W') \in L+Q$. This shows that $N \subseteq L+Q,$ which directly implies $L+Q = L+N$.

\begin {proposition}
\label {main}
The element 
$$ N(W) = \prod_{e \in \eout} U_e - \prod_{e \in \ein} U_e \in \R $$
lies in the ideal $L+Q.$
\end {proposition}

Before embarking on the proof, we present a few useful results from 
the theory of symmetric functions.

Given variables $y_1, \dots, y_m,$ the corresponding {\it elementary 
symmetric polynomials} are
$$ \ss_k (y_1, \dots, y_m) = \sum_{1 \leq i_1 < \dots < i_k \leq m} 
y_{i_1} y_{i_2} \dots y_{i_k}.$$

We also consider the {\it complete homogeneous symmetric polynomials}:
$$ \hh_k (y_1, \dots, y_m) = \sum_{1 \leq i_1 \leq \dots \leq i_k 
\leq m} y_{i_1} y_{i_2} \dots y_{i_k}.$$

We set formally $\ss_0(y_1, \dots, y_m) = \hh_0(y_1, \dots, y_m) =1,$ 
and $\ss_k(y_1, \dots, y_m) = \hh_k(y_1, \dots, y_m)  = 0$ for $k$ 
negative. Observe that $\ss_k  (y_1, \dots, y_m) = 0$ for $k > m.$

\begin {lemma}
\label {one}
For any $n \geq 1,$ we have
$$ \sum_{k+l=n} (-1)^l \ss_k(y_1, \dots, y_m) \hh_l(y_1, \dots, y_m) 
= 0.$$
\end {lemma}

\begin {proof}
For some indices $1 \leq i_1 < \dots < i_s \leq m$ and exponents $r_1, 
\dots, r_s > 0$ such that $\sum r_j =n$, the monomial $y_{i_1}^{r_1} \dots y_{i_s}^{r_s}$ 
appears in the term 
$$ \ss_k(y_1, \dots, y_m) \hh_l(y_1, \dots, y_m) $$
exactly $\binom{s}{k}$ times. In the alternating sum of these terms 
which appears in the statement of the lemma, the coefficient of this 
monomial is therefore:
$$ \sum_{k=0}^s (-1)^{n-k} \binom{s}{k} = 0.$$ \end {proof}

\begin {lemma}
\label {two}
For variables $y_1, \dots, y_n; z_1, \dots, z_m,$ we have
\begin {equation}
\label {eq:2}
 \sum_{k+l=n} (-1)^l \ss_k(y_1, \dots, y_n, z_1, \dots, z_m) 
\hh_l(z_1, \dots, z_m)
= \ss_n (y_1, \dots, y_n).
\end {equation}
\end {lemma}

\begin {proof}
Note that
$$ \ss_k(y_1, \dots, y_n, z_1, \dots, z_m) = \sum_{i+j=k}\ss_i(y_1, 
\dots, y_n) \ss_j(z_1, \dots, z_m).$$

Thus, after reordering terms, the left hand side of (\ref{eq:2}) can be 
written as:
$$ \sum_{i=0}^n \Bigl( \ss_i(y_1, \dots, y_n) \cdot \sum_{j+l = n-i} 
(-1)^l \ss_j(z_1, \dots, z_m)\hh_l(z_1, \dots, z_m) \Bigr).$$

By Lemma~\ref{one}, the interior sum is zero unless $n-i=0,$ so 
we are only left with the term $ \ss_n (y_1, \dots, y_n).$
\end {proof}

\medskip

\begin {proof}[Proof of Proposition~\ref{main}]
For simplicity, we assume that $S=\hat \Gamma$ has no two-valent vertices. By Lemma~\ref{lem:torres} (b), this results in no loss of generality.

The partial braid graph $S$ consists of $\Gamma$ together with some 
strands used to take the braid closure. Let us 
denote the variables corresponding to those strands by $B_1, \dots, 
B_k$ (that is, each $B_i$ is the same as $U_e$ for the respective strand $e$). See 
Figure~\ref{fig:braid} for an example. 

\begin {figure}
\begin {center}
\input {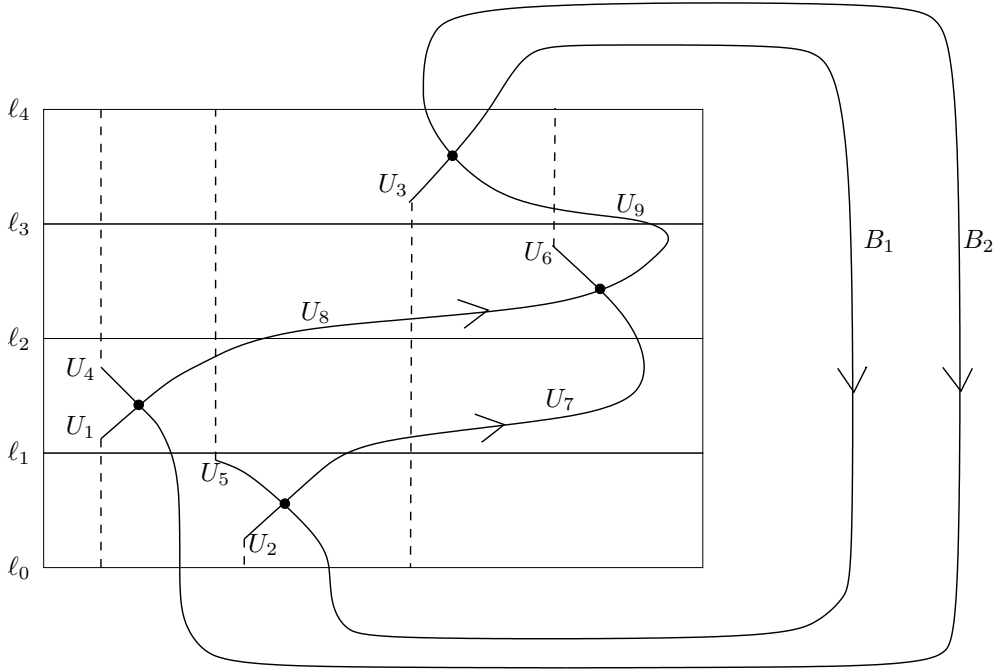}
\caption {The square $D$ is partitioned into strips and has the
dashed curves added. On each edge we mark a corresponding variable.}
\label {fig:full}
\end {center}
\end {figure}

We split the square $D = [0,1]\times [0,1]$ into horizontal strips by parallel lines, such that 
each crossing in $W$ lies in exactly one strip. We denote the parallel 
lines, including the bottom and the top of the square, by $\ell_0, 
\dots, \ell_m,$ in this order from bottom to top, such that the 
crossing 
$p_i \in W = \{p_1, \dots, p_m\}$ lies in the strip between 
$\ell_{i-1}$ and $\ell_i.$

We extend each edge $e \in \ein$ by a dashed curve going vertically 
down to the bottom of the square $D$, and each $e \in \eout$ by a 
dashed curve going vertically up to the top of the square. The 
intersections between dashed curves, or between a dashed curve and some 
part of the braid, are irrelevant.

After these constructions, the example in Figure~\ref{fig:braid} gets 
transformed into Figure~\ref{fig:full}.

Now each horizontal line $\ell_i$ intersects a total of $n+k$ curves 
(either regular edges or their dashed continuations), where $n$ is the 
the cardinality of $\ein$, which is the same as the cardinality of $\eout$. We 
denote by $F_i$ the multiset of edges intersecting $\ell_i,$ making no 
distinction between an edge and its dashed continuation. A multiset is the generalization of a set, where elements can have higher multiplicities. In our setting, the strands labeled by $B$'s may have higher multiplicities in $F_i$. If so, we want to count each 
edge with its corresponding multiplicities; for example, when we write 
$\sum_{e \in F_i}$, we count $e$ as many times as it appears in $F_i.$

The strip between $\ell_{i-1}$ and $\ell_i$ contains the crossing 
$p_i;$ we denote the two edges going out of $p_i$ by $a_i$ and $b_i$, 
and the two going in by $c_i$ and $d_i.$ If we eliminate $a_i$ and 
$b_i$ from $F_i$ (only once though, in case they appear multiple times) 
we obtain a multiset $G_i$, which is the same as the one obtain from 
$F_{i-1}$ by eliminating $c_i$ and $d_i$ (again, only once).

The claim of the proposition will follow from the following identity:
\begin {eqnarray}
\label {eq:ide}
 \prod_{e \in \eout} U_e - \prod_{e \in \ein} U_e &=& \sum_{i=1}^m 
L(p_i) \Bigl( \sum_{j\geq 0} (-1)^j \ss_{n-1-j} \bigl ( \{U_e | e \in 
G_i\}\bigr) \hh_j (B_1, \dots, B_k) \Bigr) \notag \\
& \ & + \sum_{i=1}^m
Q(p_i) \Bigl( \sum_{j\geq 0} (-1)^j \ss_{n-2-j} \bigl ( \{U_e | e \in
G_i\}\bigr) \hh_j (B_1, \dots, B_k) \Bigr).
\end {eqnarray}

For concreteness, let us write down the relation (\ref{eq:ide}) 
in the example pictured in Figure~\ref{fig:full}:

\medskip
$ U_4U_5U_6 - U_1U_2U_3  = $
\begin {eqnarray*}
&=& (U_5 + U_7 - U_2 - B_1) \Bigl( (U_1B_2 + U_3B_2 + U_1U_3) - (B_1 + 
B_2) (U_1 + U_3 + B_2) + (B_1^2 + B_1B_2 + B_2^2) \Bigr) \\
&+& (U_4 - U_8 - U_1 - B_2) \Bigl( (U_3U_5 + U_3U_7 + U_5U_7) - (B_1 +
B_2) (U_3 + U_5 + U_7) + (B_1^2 + B_1B_2 + B_2^2) \Bigr) \\
&+& (U_6 + U_9 -U_8 - U_7) \Bigl( (U_3U_4 + U_3U_5 + U_4U_5) - (B_1 +
B_2) (U_3 + U_4 + U_5) + (B_1^2 + B_1B_2 + B_2^2) \Bigr) \\
&+& (B_1 + B_2 -U_3 - U_9) \Bigl( (U_4U_5 + U_4U_6 + U_5U_6) - (B_1 +
B_2) (U_4 + U_5 + U_6) + (B_1^2 + B_1B_2 + B_2^2) \Bigr) \\
&+& (U_5U_7 - U_2B_1) \Bigl( (U_1 + U_3 + B_2) - (B_1+B_2) \Bigr) \\
&+& (U_4U_8 - U_1B_2 ) \Bigl( (U_3 + U_5 + U_7) - (B_1+B_2) \Bigr) \\
&+& (U_6U_9 - U_8U_7) \Bigl( (U_3 + U_4 + U_5) - (B_1+B_2) \Bigr) \\
&+& (B_1B_2 - U_3U_9) \Bigl( (U_4 + U_5 + U_6) - (B_1+B_2) \Bigr). \\
\end {eqnarray*}

In order to prove \eqref{eq:ide} in general, we start by observing that

$$ \ss_{n-j}  \bigl ( \{U_e | e \in F_i\}\bigr)  = \ss_{n-j}  \bigl ( 
\{U_e | e \in G_i\}\bigr) + (U_{a_i} + U_{b_i}) \ss_{n-j-1}  \bigl (  
\{U_e | e \in G_i\}\bigr) + U_{a_i}U_{b_i} \ss_{n-j-2}  \bigl (
\{U_e | e \in G_i\}\bigr); $$

$$ \ss_{n-j} \bigl ( \{U_e | e \in F_{i-1}\}\bigr)  = \ss_{n-j} \bigl ( 
\{U_e | e \in G_i\}\bigr) + (U_{c_i} + U_{d_i}) \ss_{n-j-1} \bigl ( 
\{U_e | e \in G_i\}\bigr) + U_{c_i}U_{d_i} \ss_{n-j-2} \bigl ( \{U_e | 
e \in G_i\}\bigr). $$
\medskip

Subtracting the second relation from the first, we get 

$$ L(p_i) \ss_{n-1-j} \bigl ( \{U_e | e \in G_i\}\bigr) + Q(p_i)   
\ss_{n-j-2} \bigl( \{U_e | e \in G_i\}\bigr) = \ss_{n-j} \bigl ( \{U_e 
| e \in 
F_{i}\}\bigr) -  \ss_{n-j}  \bigl ( \{U_e | e \in F_{i-1}\}\bigr). $$
\medskip
 
Thus, after changing the order of summation, the right hand side of 
(\ref{eq:ide}) can be re-written as
$$  \sum_{j\geq 0} (-1)^j \sum_{i=1}^m \Bigl( \ss_{n-j} \bigl ( \{U_e | 
e \in F_{i}\}\bigr) -  \ss_{n-j}  \bigl ( \{U_e | e \in F_{i-1}\}\bigr) 
\Bigr) \cdot \hh_j (B_1, \dots, B_k)  $$

\begin {equation}
\label {cr}
= \sum_{j\geq 0} (-1)^j \Bigl( \ss_{n-j} \bigl ( \{U_e |
e \in F_m \}\bigr) -  \ss_{n-j}  \bigl ( \{U_e | e \in F_0 \}\bigr) 
\Bigr) \cdot \hh_j (B_1, \dots, B_k).
\end {equation}

Note that $F_0$ is the union of $\ein$ with the set of strands $\{B_1, 
\dots, B_k\}.$ Applying Lemma~\ref{two}, we obtain the identity
$$  \sum_{j\geq 0} (-1)^j  \ss_{n-j} \bigl ( \{U_e |
e \in F_0 \}\bigr) \cdot \hh_j (B_1, \dots, B_k) = \ss_n  \bigl ( \{U_e 
| e \in \ein \}\bigr). $$

Similarly, $F_m$ is the union of $\eout$ with the set of strands 
$\{B_1, \dots, B_k\},$ hence
$$  \sum_{j\geq 0} (-1)^j  \ss_{n-j} \bigl ( \{U_e |
e \in F_m \}\bigr) \cdot \hh_j (B_1, \dots, B_k) = \ss_n  \bigl ( \{U_e
| e \in \eout \}\bigr). $$

Putting these together, we get that the expression in (\ref{cr}) equals
$$ \ss_n  \bigl ( \{U_e | e \in \eout \}\bigr) - \ss_n  \bigl ( \{U_e
| e \in \ein \}\bigr) = \prod_{e \in \eout} U_e - \prod_{e \in \ein} 
U_e,$$
as desired. \end {proof}

\subsection {Total braid graphs}
\label {sec:noas}
Assumption~\ref{as} in the definition of partial braid graph $S$ required that $S$ has some loose ends. Let us define a {\em total braid graph} $S$ to be the braid closure $S = \hat \Gamma$ of an open braid  graph $\Gamma$, with the property that $S$ has no loose ends. 

Interestingly, the proof of Theorem~\ref{thm:2} did not use Assumption~\ref{as}. However, this assumption is needed for the equality of the higher $\Tor$ groups in Conjecture~\ref{conj2}. To see this, consider the total braid graph $S$ from Figure~\ref{fig:total}. We have:
\begin{eqnarray*}
 \R &=& \zz[U_1, U_2, U_3, U_4], \\
L &=& (U_2 - U_4), \\
N &=& (U_2 - U_4), \\
Q &=& (U_1(U_2 - U_4), U_3(U_2 - U_4)).
\end {eqnarray*}

\begin {figure}
\begin {center}
\input {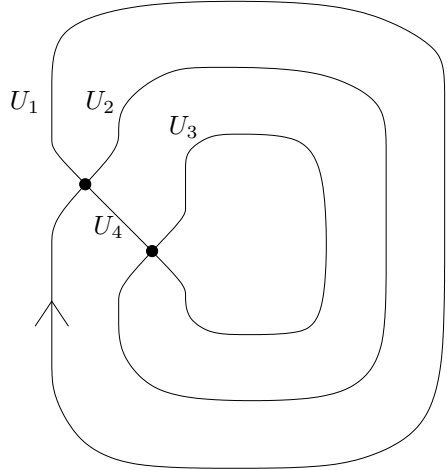}
\caption {A total braid graph.}
\label {fig:total}
\end {center}
\end {figure}

Let $\R' = \R/(U_2 - U_4).$ Then:
 $$\Tor_1(\R/L, \R/N) \cong \R',$$
 whereas
 $$ \Tor_1(\R/L, \R/Q) \cong  \R'\langle x, y \rangle / (U_1 x - U_3 y).$$
These $\R$-modules are not isomorphic, so Conjecture~\ref{conj2} fails for $S$.

\subsection {Computer experimentation}
Conjecture~\ref{conj2} (and its graded refinement, Conjecture~\ref{conj3}) can be checked for many partial braid graphs using the program {\em Macaulay2} \cite{M2}. The program gives presentations of the $\R$-modules $\Tor_i(\R/L, \R/N)$ and $\Tor_i(\R/L, \R/Q)$. For small graphs, it is visible that the modules are isomorphic. However, in general there is no simple way of checking whether two presentations give isomorphic modules. For larger graphs, we settled for verifying Conjecture~\ref{conj3} at the level of Hilbert series.

Precisely, for a homogeneous module $M$, let us denote by $r_d(M)$ the rank of the degree $d$-graded piece of $M$. {\em Macaulay2} automatically grades polynomial rings by letting each variable have grading $1$. (On $\R$, this corresponds to half of the grading $\gr_q$ from Subsection~\ref{sec:conjgr}.) Given a partial braid graph $S$ and $i \geq 0$, we consider the Hilbert series
$$ n_i(S)  = \sum_{d \geq 0} T^d \cdot r_d(\Tor_i(\R/L, \R/N))$$
 and
 $$ q_i(S) = \sum_{d \geq 0} T^d \cdot r_d(\Tor_i(\R/L, \R/Q)).$$

We know from Theorem~\ref{thm:2} that $q_0 = n_0$. Conjecture~\ref{conj3} would imply that 
\begin {equation}
\label {eq:hilbert}
q_i(S) = T^i \cdot n_i(S), \ \ \text{for all } i \geq 0.
\end {equation}

In practice, the relation \eqref{eq:hilbert} is much easier to check than the existence of module isomorphisms. In view of Lemma~\ref{lem:bound}, it makes sense to only look at the values $i \leq k$, where $k$ is the number of strands used to close up the partial braid $S$. (We know that $q_i = 0$ for $i > k$, and expect this to also be true for $n_i$.) For example, for the partial braid graph from Figure~\ref{fig:braid}, we find that
\begin {eqnarray*}
 q_0 = n_0 &=& (1+3T+2T^2-2T^3) / (1-T)^{4} , \\
 q_1 = T \cdot n_1 &=& T^4(3+T)/(1-T)^4, \\
 q_2 = T^2 \cdot n_2 &=& 0.
\end {eqnarray*}

 We focused most of our computer experiments on complete resolutions of decorated braid projections, where all crossings are singularized; that is, we took a braid $b$ on $k+1$ strands, singularized all its crossings, and closed up $k$ of the strands (all but the leftmost one) to get $S$. We verified that \eqref{eq:hilbert} holds (for $i \leq k$) for all connected $S$ of this form, with $k \leq 3$ and at most $7$ crossings.

\bibliographystyle{custom}
\bibliography{biblio}

\end{document}